\documentclass[dvips,preprint]{imsart}
\usepackage{mathrsfs}
\usepackage{color}
\usepackage{ amsmath, amsfonts, amssymb, amsthm, amscd}
\usepackage[T1]{fontenc}
\usepackage[english]{babel}
\def\Dy#1{\Frac{\partial #1}{\partial y}}

\def\Dy_1y_1#1{\Frac{\partial^2 #1}{\partial y_1^2}}

\def\reff#1{{\rm(\ref{#1})}}
\def\be{\begin{eqnarray}}
\def\ee{\end{eqnarray}}
\def\b*{\begin{eqnarray*}}
\def\e*{\end{eqnarray*}}

\newtheorem{Theorem}{Theorem}[part]
\newtheorem{Definition}{Definition}[part]
\newtheorem{Proposition}{Proposition}[part]

\newtheorem{Lemma}{Lemma}[part]
\newtheorem{Corollary}{Corollary}[part]
\newtheorem{Remark}{Remark}[part]
\newtheorem{Example}{Example}[part]

\makeatletter \@addtoreset{equation}{section}

\@addtoreset{Definition}{section}

\@addtoreset{Theorem}{section}

\@addtoreset{Proposition}{section}

\@addtoreset{Assumption}{section}

\@addtoreset{Corollary}{section}

\@addtoreset{Lemma}{section}

\@addtoreset{Remark}{section}

\@addtoreset{Example}{section}

\makeatother \makeatletter

\def \dint{\displaystyle\int}
\def \Frac{\displaystyle\frac}

\def \be{\begin{eqnarray}}
\def \ee{\end{eqnarray}}
\def \b*{\begin{eqnarray*}}
\def \e*{\end{eqnarray*}}

\def \E{\mathbb{E}}
\def \F{\mathbb{F}}

\def \N{\mathbb{N}}

\def \R{\mathbb{R}}

\def \[{[\,\!\![}
\def \]{]\,\!\!]}

\def \1{{\bf 1}}
\def \esssup{{\rm esssup}}
\def \essinf{{\rm essinf}}

\def \ep{\hbox{ }\hfill$\Box$}

\def\reff#1{{\rm(\ref{#1})}}

\def\Cc{{\cal C}}
\def\Dc{{\cal D}}
\def\Ec{{\cal E}}
\def\Fc{{\cal F}}

\def\Hc{{\cal H}}

\def\k{\kappa}

\def\Sc{{\cal S}}

\def\Uc{{\cal U}}

\def\Qc{{\cal Q}}

\def\Rc{{\cal R}}

\begin{document}
\begin{frontmatter}
\title{Optimal stochastic control problem under model uncertainty with non-entropy  penalty }
\date{}
\runtitle{}

\author{\fnms{Wahid}
 \snm{FAIDI}\corref{}\ead[label=e1]{wahid.faidi@enit.rnu.tn}}
\address{University of Tunis El Manar\\
 Laboratoire de Mod\'elisation
 Math\'ematique et Num\'erique\\
 dans les Sciences  de l'Ing\'enieur, ENIT\\
 \printead{e1}
 }

\author{\fnms{Anis}
 \snm{MATOUSSI}\corref{}\ead[label=e2]{anis.matoussi@univ-lemans.fr}}
 \thankstext{t3}{Research partly supported by the Chair {\it Financial Risks} of the {\it Risk Foundation} sponsored by Soci\'et\'e G\'en\'erale, the Chair {\it Derivatives of the Future} sponsored by the {F\'ed\'eration Bancaire Fran\c{c}aise}, and the Chair {\it Finance and Sustainable Development} sponsored by EDF and Calyon }
\address{
 Universit\'e du Maine \\
 Institut du Risque et de l'Assurance\\
Laboratoire Manceau de Math\'ematiques\\\printead{e2}
 }

\author{\fnms{Mohamed}
 \snm{MNIF}\corref{}\ead[label=e3]{mohamed.mnif@enit.rnu.tn}}
\thankstext{t4}{This work was partially supported by the research project GEANPYL
of  FR 2962 du CNRS Math\'ematiques des Pays de Loire}
\address{
University of Tunis El Manar \\
Laboratoire de Mod\'elisation
 Math\'ematique et Num\'erique\\
 dans les Sciences  de l'Ing\'enieur, ENIT\\\printead{e3}}

\runauthor{W. Faidi,  A. Matoussi, M. Mnif}

%

\begin{abstract}
In this paper, a stochastic control problem under model uncertainty with general penalty term is studied. Two types of penalties are considered. The first one is of type f-divergence penalty treated in the general framework of a continuous filtration. The second one called consistent time penalty studied in the context of a Brownian filtration.
In the case of consistent time penalty, we characterize the value process of our stochastic control problem as the unique solution of a class of quadratic backward stochastic differential equation with unbounded terminal condition.
\end{abstract}

\vspace{7mm}

\noindent {\bf Key words~:} Robust stochastic control, model uncertainty, Knightian uncertainty, Backward Stochastic Differential Equations,  utility maximization.

\vspace{5mm}

\noindent {\bf MSC Classification (2000)~:} 92E20, 60J60, 35B50.
\end{frontmatter}
\section{Introduction}
The problem of random systems under model uncertainty constitutes an important topic of research. It occurs e.g. in risk management (utility maximization problems for economic agents) and pricing/hedging (cheapest superreplication of a contingent claim). Model uncertainty is a major concern for practical applications since one has an imperfect knowledge of model (unknown drift, unknown volatility and correlation matrices, unknown jumps modeling. . . ). The utility maximization problem is concerned with optimal investment faced by an economic agent who has the opportunity to invest in a financial market consisting of a riskless asset and one risky as- set. Following the seminal work of Von Neumann-Morgenstern \cite{{VnM44}} assuming a utility function for representing the agent preference, with a given probability measure reflecting his views, the wealth is to be optimized over a set of admissible strategies (see \cite{Mer71} for the solution under log-normal assumptions about the risky asset).
\paragraph{Litterature review of the robust utility maximization problem.}
In the above formulations, a probability measure is fixed, meaning that the agent knows the "historical" probability for describing the state process dynamics. In reality, the agent may have some uncertainty on this probability, leading to several objective probability measures to consider. Preliminary results in the literature has been obtained in the case of dominated sets, namely with an objective reference probability measure (like drift uncertainty \cite{GS89}). In \cite{AHS03, HS01, HSTW06} the authors have discussed the basic problem of Robust Utility Maximization (RUM), penalized by a relative entropy term of the model uncertainty with respect to the reference probability measure, see also \cite{BMS05}.  On the other hand in \cite{HIM06, RElk00}, the authors have linked BSDEs with quadratic growth to the problem of utility maximization under constraints.
Nevertheless, as in the utility maximization problem described above, this approach can only deal with problems where the law of the underlying risk factors is assumed to be known or to belong to a dominated set of probability measures. Note that  the case when the volatilities/correlations are not precisely known, are not covered by our study.

 An other approach used to tackle the RUM is the duality theory  (see \cite{QUEN04, SCW05, SC08}).  Wittmuss \cite{WIT06} have extended these results to cover also the cases of consumption-investment strategies and random endowment (see also  \cite{BUR05} for some earlier results in that direction). 

\paragraph{Main contributions.}
 In this paper,  we  work in the context of  Hansen and Sargent  (\cite{AHS03, HS01, HSTW06})  and Bordigoni, Matoussi and Schweizer \cite{BMS05}  to characterize the optimal preferences in terms of forward backward stochastic differential equation in the non-entropic penalty  case. Our approach is based on the dynamic programming  Bellman principle, and in the case of  time consistent penalty,  the value function is characterized  as the solution of a  class of  backward stochastic differential equations with quadratic generator and unbounded terminal condition. Our methodology is different from the dual approach  which is based on the Legendre-Frenchel transformation. Particularly, Föllmer and Gundel \cite{FG06}  studied  the $f$-divergence minimal martingale  measure  problem and the connexion with the theory  of optimal portfolio choice, in particular to robust utility maximization problem.  
Precisely,  we  study the following problem
\begin{equation*}
 \inf\limits_{Q\in\Qc}E_Q [\Uc^{\delta}_{0,T}+ \beta \Rc^{\delta}_{0,T}(Q)],
 \end{equation*}
 where
  $$\Uc^{\delta}_{t,T}:=\alpha \dint_t^TS_s^\delta U_sds+\bar{\alpha}S_T^\delta \bar{U}_T, $$
 with  $\alpha,\bar{\alpha}$ are two non negative parameters, $ \beta \in(0,+\infty)$,  $(U_t)_{0\leq t \leq T}$ a progressively measurable process, $\bar{U}_T$ a random variable $\Fc_T$-measurable and  $S^\delta$ is the discounting process defined by
$S_t^\delta := \exp(-\dint_0^t\delta_sds);0\leq t \leq T$
where $(\delta_t)_{0\leq t \leq T}$ is a progressively measurable  and bounded process. $\Rc_{t,T}^\delta(Q)$ denotes a penalty term which is written as a sum of a penalty rate and a final penalty. The cost functional
\begin{equation*}
c(w,Q):=\Uc_{0,T}^\delta+\beta \Rc_{0,T}^\delta(Q)
\end{equation*}
consists of two terms. The first is a $Q$-expected discounted utility with discount rate $\delta$, utility rate $U_s$ at time $s$ and terminal utility $\bar{U}_T$ at time $T$.  Usually, $(U_s)_{s \geq 0}$ comes from consumption and $\bar{U}_T$ is related to the terminal wealth. Note that we do not specify any financial model and  we assume only  that   $\bar{U}_T$  and $(U_s)_{s \geq 0}$  belong to Orlicz spaces. Therefore our approach covers a general setting. The second term, which depends only on $Q$, is a penalty term which can be interpreted as being a kind of "distance" between $Q$ and the historical probability measure $P$.
  The role of proportionality parameter  $\beta$ is to measure the degree of confidence of the decision maker in the reference model $P$, or, in other words, the concern for the model erroneous specification.  The higher value of $\beta$ corresponds to more confidence.
\\In this paper we studied two classes of penalties.
The first class is the $f$-divergence penalty introduced by Cizar \cite{Ciz} given in our framework by:
\begin{equation}\label{f-divergence}
\Rc_{t,T}^\delta(Q):=\dint_t^T\delta_s\frac{S_s^\delta}{S_t^\delta} \frac{Z_t^Q}{Z_s^Q}f(\frac{Z_s^Q}{Z_t^Q})ds+\frac{S_T^\delta}{S_t^\delta} \frac{Z_t^Q}{Z_T^Q}f(\frac{Z_T^Q}{Z_t^Q}); \forall 0\leq t \leq T.
\end{equation}
where $f$ is a convex function.
In this case,  set $\mathcal{Q}$ consists of all models $Q$  absolutely continuous with respect to $P$ whose density process (with respect to $P$) $Z^Q$ satisfies $\E_P[\displaystyle{f(Z^Q_T)}]<+\infty.$
\\The second  class called consistent time penalty studied in the context of a  filtration generated by a Brownian motion $(W_t)_{0\leq t \leq T}.$
\begin{equation}\label{consistent}\Rc_{t,T}^\delta(Q):=\dint_t^T\delta_s\frac{S_s^\delta}{S_t^\delta} (\int_t^s h(\eta_u)du)ds+\frac{S_T^\delta}{S_t^\delta}\int_t^T h(\eta_u)du ; \forall 0\leq t \leq T
\end{equation}
where $h$ is a convex function and the density process of $Q^\eta$ with respect to $P$ can be written: $$\frac{dQ^\eta}{dP}=\Ec(\int_0^.\eta_u dW_u).$$
In this case, the set $\mathcal{Q}$ is formed by all models $Q^\eta $ absolutely continuous with respect to $P$ such that $\E_{Q^\eta}[\dint_0^Th(\eta_s)ds]<+\infty.$ Using HJB equation technics, Schied \cite{SC08} studied the same problem when $\delta$ is constant and the process $\eta$ takes values in a compact convex set in $\mathbb{R}^2$.  Finally,  more recently  Laeven and Stadje \cite{LS12} have studied the case of consistent time penalty by using  another proof of the existence result and  by assuming a bounded terminal condition. \\

The paper is organized as follows.  The robust utility problem where the penalty is modeled by the $f$-divergence is studied in the next section . In section 3, we present  the class of consistent time penalty. In particular,  we characterize in this case  the value process for our control problem as the unique solution of a generalized class of  quadratic BSDEs. In subsection \ref{exampleprotfolio}, we study a portfolio investment choice problem where the uncertainty is modeled via a consistent time penalty. In this case the RUM problem is equivalent to the classical expected utility maximization when the underlying model is known and the utility function could be seen as an modified one. When the utility is a power function and the risky asset follows the Black-Scholes model, we give an explicit solution for the optimal strategy $\pi^*$ of investment in the risky asset. The originality of the formula comes from the explicit dependence of the entrpoic penalty term in the strategy $\pi^*.$   Finally, we give  some technical results in the  Appendix.

\section{Class of f-divergence penalty}
\subsection{The setting}
This section gives a precise formulation of our optimization problem and introduces our notations for later use.
We start with a filtered probability space $(\Omega,\Fc, \F,P)$ over a finite time horizon $T\in (0,+\infty).$
The filtration $\F = (\Fc_t)_{0\leq t \leq T}$ satisfies the usual conditions of right-continuity and $P$-completeness.
\\For any probability measure $Q\ll P$ on $\Fc_T$, the density process of $Q$ with respect to $P$ is the RCLL $P$-martingale $Z^Q = (Z^Q_t )_{0\leq t \leq T}$ with
$$Z^Q_t = \frac{dQ}{dP}\mid_{\Fc_t}= \E_P[\frac{dQ}{dP}\mid{\Fc_t}],\;\; \textrm{for all} \;\; 0\leq t \leq T.$$
Since $Z^Q$ is closed on the right by $Z^Q_t = \frac{dQ}{dP}\mid_{\Fc_t}$, $Z^Q$  can be identified with $Q$.
\\Let $f : [0,+\infty) \mapsto \R$ be a continuous and strictly convex function  satisfying:
\begin{enumerate}
\item[\textbf{(H1)}] $f(1)=0.$
\item[\textbf{(H2)}] There is a constant $\kappa \in \R_+$ such that $f(x)\geq -\kappa, \;\;\textrm{for all}\;\; x\in (0,+\infty).$
\item[\textbf{(H3)}] $\lim\limits_{x \mapsto +\infty}\dfrac{f(x)}{x}=+\infty.$
\item[\textbf{(H4)}]   $f$ is differentiable on $(0,+\infty)$ and $f'(0)=\lim\limits_{x\rightarrow 0^+}f'(x)=-\infty.$
\end{enumerate}
\begin{Remark}
Note that the Assumption \textbf{(H2)} can be relaxed . Indeed, since $f$ is convex, then $f(x)\geq f'(1)(x-1)+f(1)=f'(1)x-f'(1)$ and so $f(x)-f'(1)x \geq -f'(1).$ Therefore, one could replace  $f(x)$ by $f(x)-f'(1)x, \forall x\geqslant 0 .$
\end{Remark}
Our basic goal is to
\begin{equation}\label{min1} \textrm{minimize the functional}\; Q \mapsto \Gamma(Q) := \E_Q[c(.,Q)]\end{equation}
over a suitable class of probability measures $Q \ll P$ on $\Fc_T$, where the cost functional $c(.,Q)$ is defined by
\begin{equation}\label{costfunction}
c(w,Q):=\Uc_{0,T}^\delta+\beta \Rc_{0,T}^\delta(Q),
\end{equation}
with the utility term given by
 $$\Uc^{\delta}_{0,T}:=\alpha \dint_0^TS_s^\delta U_sds+\bar{\alpha}S_T^\delta \bar{U}_T, $$
and the penalty term is
$$\Rc_{0,T}^\delta:=\dint_0^T\delta_sS_s^\delta \frac{f(Z_s^Q)}{Z_s^Q}ds+S_T^\delta \frac{f(Z_T^Q)}{Z_T^Q},\;\; \textrm{for all} \;\; 0\leq t \leq T.$$

\begin{Definition}
For a convex function $\varphi$ we define the following functional spaces:
\\$L^{\varphi} $ is the space of all ${\cal{ F}}_T $ measurable random variables $X $ with
 $$E_P\left[\varphi\left(\gamma \vert X  \vert\right)\right]<\infty, \qquad \hbox{ for all }
 \gamma>0,$$
 $D^{\varphi}_0$ is the space of all progressively measurable processes $X={(X_t)}_{0\le t\le T}$ with
 $$E_P\left[\varphi\left(\gamma {~ \rm{ess}\sup}_{0\le t\le T}|X_t|\right)\right]<\infty , \qquad
 \hbox{ for all } \gamma>0,$$
 $D^{\varphi}_1$ is the space of all progressively measurable processes $X={(X_t)}_{0\le t\le T}$ such that
$$E_P\left[\varphi\left(\gamma\int_0^T |X_s|ds\right)\right]<\infty,  \qquad \hbox{ for all } \gamma >0.$$
\end{Definition}
\begin{Remark}
The spaces, $L^{\varphi} $, $D^{\varphi}_0$ and $D^{\varphi}_1$  are called Orlicz spaces. For more detail see \cite{Rao}
\end{Remark}
\begin{Definition}
For any probability measures $Q$ on $(\Omega,\Fc)$, we define the f-divergence of  $Q$ with respect to $P$ by:
$$
d(Q|P):=\left\{\begin{array}{cc}
             \E_P[f(\frac{dQ}{dP}|_{\Fc_T})] & \textrm{if} \; Q \ll P \;\textrm{on}\; \Fc_T \\
             +\infty & \textrm{otherwise.}
           \end{array}\right.
$$
\\We denote by $\Qc_f$ the space of all probability measures $Q$ on $(\Omega,\Fc)$ with $Q \ll P$ on $\Fc_T,$  $Q = P$ on $\Fc_0$ and $d(Q|P) < +\infty.$ The set $\Qc_f^e$ is defined as follows $$\Qc_f^e:=\{Q \in \Qc_f | Q\approx P \; \textrm{on} \; \Fc_T \}.$$
\end{Definition}
\begin{Example}
\begin{enumerate}
\item The relative entropy: If $f(x):=x\ln x,$ then $d(Q|P)$ is called relative entropy and is denoted by $H(Q|P).$
\item The Bergman divergence: it matches to the function $f(x):=x\ln x-x.$
\end{enumerate}
\end{Example}
The conjugate function of $f$ on $\R_+$ is defined  by:
\begin{equation}\label{conjugate function}
f^{\ast }(x)=\sup\limits_{y>0}{(xy-f(y))}.
\end{equation}
$f^{\ast }$ is a convex function, nondecreasing, nonnegative and satisfies:
\begin{equation}\label{convex inequality1}
xy\leq f^{\ast }(x)+f(y) ,\;\;\textrm{for all} \;\; x\in \R_+ \;\; \mbox{and}\;\; y>0,
\end{equation}
and also
\begin{equation}\label{convex inequality2}
xy\leq\frac{1}{\gamma} [f^{\ast }(\gamma x)+f(y)] ,\;\;\textrm{for all}\;\; x\in \R_+, \;\; \gamma>0\;\; \mbox{and} \;\; y>0.
\end{equation}
For a precise formulation of (\ref{min1}), we now assume:
\begin{enumerate}
\item[(\textbf{\textrm{A1}})]$\delta $ is positive and bounded by $\left\Vert
\delta \right\Vert _{\infty }.$
\item[(\textbf{\textrm{A2}})] Process $U$ belongs to $D_1 ^{f^*}$ and the random variable $\bar{U}_T$ is in $L^{f^*}$
\end{enumerate}
\begin{Remark}\label{boundness of fdiv}
\begin{enumerate}
\item[1-] Assumption (\textbf{H2}) implies that :
\begin{equation}\label{absolute}
\mid f(x)\mid \leq f(x)+2\kappa, \quad \forall x \geq 0.
\end{equation}
\item[2-] In the case of entropic penalty, we have $f(x)=x\ln(x)$ and then $f^{\ast }(x)=\exp({x-1}). $ As in Bordigoni, Matoussi and Schweizer \cite{BMS05}, the integrability conditions are formulated as
$$
\mathbb{E}_P\left[\displaystyle \exp({\lambda \int_{0}^{T}|U(s)|ds})\right] <+\infty \;\; \textrm{and} \;\;
\mathbb{E}_P\left[ \displaystyle \exp({\lambda  |\bar{U}_T|})\right] <+\infty\;\;\textrm{for all} \;\; \lambda>0.
$$
\end{enumerate}
\end{Remark}
\subsection{Existence of  optimal probability measure}
The main result of this section is to prove that the problem (\ref{min1}) has a unique solution $Q^*\in \Qc_f.$ Under some additional assumptions, we prove that $Q^*$ is equivalent to $P$. This is proved for a general filtration $\F$.
This section begins  by establishing some estimates for later use.
\begin{Proposition}\label{prop1}
Under assumptions \textbf{(A1),(A2)} and \textbf{(H2)} , we have:
\begin{enumerate}
\item $c(.,Q) \in L^1(Q)$,
\item $\Gamma (Q)\leq C(1+d(Q|P))$, where $C$  is a positive constant  depending only on $\alpha ,\bar{\alpha} ,\beta ,\delta ,T,(U_s)_{s\in[0,T]}$ and $\bar{U}_T.$
\end{enumerate}
In particular $\Gamma(Q)$ is well-defined and finite for every $Q\in \Qc_f$.
\end{Proposition}
 \begin{proof}
\begin{enumerate}
\item We first prove that for all $ Q \in  Q_f$, $c(.,Q)$ belongs to $ L^1(Q)$. Set  $R:=\alpha \int_{0}^{T}|U_s|ds+\bar{\alpha}|\bar{U}_T|$, we get
\begin{equation*}
 |Z_T^Qc(.,Q)| \leq  Z_T^Q R +\left \Vert \delta \right\Vert _{\infty } Z_T^Q \int_{0}^{T}\left\vert \frac{f(Z_{s}^{Q})}{Z_s^Q}\right\vert ds+\left\vert f(Z_{T}^{Q})\right\vert.
\end{equation*}
By  the estimate  (\ref{convex inequality1}), we have $Z_{T}^{Q}R \leq f(Z_{T}^{Q})+f^{\ast }(R).$
From assumption \textbf{(A2)}, the variable random $f^{\ast }(R)$ is in $L^1(P)$ and from Remark \ref{boundness of fdiv}, we get that   for all $ Q \in  Q_f$, $ f(Z_T^Q)$ belongs to $  L^1(P).$
It remains to show that   $Z_T^Q \displaystyle\int_{0}^{T}\left\vert \frac{f(Z_{s}^{Q})}{Z_s^Q}\right\vert ds $ belongs to $ L^1(P).$
By Tonelli-Fubini's Theorem, we have
\begin{equation*}
\begin{split}
\E_P\big[Z_T^Q \displaystyle\int_{0}^{T}\left\vert \frac{f(Z_{s}^{Q})}{Z_s^Q}\right\vert ds \big]&=\displaystyle\int_{0}^{T}\E_P\big[Z_T^Q \left\vert \frac{f(Z_{s}^{Q})}{Z_s^Q}\right\vert\big] ds\\&=\displaystyle\int_{0}^{T}\E_P\big[Z_s^Q \left\vert \frac{f(Z_{s}^{Q})}{Z_s^Q}\right\vert\big] ds=\displaystyle\int_{0}^{T}\E_P\big[\left\vert f(Z_{s}^{Q})\right\vert\big]ds.
\end{split}
\end{equation*}
Jensen's inequality allows
$$f(Z_{s}^{Q})=f\left( \E_{P}\left[ Z_{T}^{Q}|\mathcal{F}_{s}\right] \right) \leq \E_{P}\left[ f\left( Z_{T}^{Q}\right) |\mathcal{F}_{s}\right].$$
By taking the expectation under $P,$ we obtain
\begin{equation}\label{c1}\E_{P}\left( f(Z_{s}^{Q})\right) \leq \E_{P}\left[ f\left( Z_{T}^{Q}\right) \right]. \end{equation}
Consequently, from inequality (\ref{absolute}) we have
\begin{equation}\label{c1'}
\E_P\big[\vert f(Z_{s}^{Q})\vert\big]\leq \E_{P}\left[ f\left( Z_{T}^{Q}\right) \right]+2\kappa,
\end{equation} and so,  $s\mapsto\E_P\big[\vert f(Z_{s}^{Q})\vert\big]$ is in $L^1([0,T]).$  Whence,   $Z_T^Q
\displaystyle\int_{0}^{T}\left\vert \frac{f(Z_{s}^{Q})}{Z_s^Q}\right\vert ds $ belongs  to $ L^1(P).$
\item From the definition of $\Gamma$, we have
\begin{equation*}
 \Gamma (Q) \leq \E_{P}\left[ Z_{T}^{Q}R\right] +\beta \E_{P}\left[\left\Vert \delta \right\Vert _{\infty }\int_{0}^{T}\left\vert f(Z_{s}^{Q})\right\vert ds+\left\vert f(Z_{T}^{Q})\right\vert \right].
\end{equation*}
By the inequality (\ref{c1'}), we have
\begin{equation*}
\begin{split}
\E_{P}\left[\left\Vert \delta \right\Vert _{\infty}\int_{0}^{T}\vert f(Z_{s}^{Q})\vert ds+\vert f(Z_{T}^{Q})\vert \right]&=\left\Vert \delta \right\Vert _{\infty}\int_{0}^{T}\E_{P}\left[\vert f(Z_{s}^{Q})\vert\right] ds+\E_{P}\left[\vert f(Z_{T}^{Q})\vert \right]\\&\leq \left\Vert \delta \right\Vert _{\infty}T\Big(\E_{P}\left[ f\left( Z_{T}^{Q}\right) \right]+2\kappa\Big)+\E_{P}\left[ f\left( Z_{T}^{Q}\right) \right]+2\kappa,
\end{split}
\end{equation*}
and consequently, $$\Gamma (Q)\leq \E_{P}\left[ f^{\ast }(R)\right] +2 \kappa \beta
(\left\Vert \delta \right\Vert _{\infty }T+1)+(1+\beta \left\Vert \delta
\right\Vert _{\infty }T+\beta )d(Q|P).$$
We define the constant $C$ by:
$$C:=\max \left( E_{P}\left[ f^{\ast }(R)\right]+2\kappa\beta (\left\Vert
\delta \right\Vert _{\infty }T+1),(1+\beta \left\Vert \delta \right\Vert
_{\infty }T+\beta )\right) . $$
From assumptions (\textbf{A1})-(\textbf{A2}), $C$  is finite, positive and answers the question.
\end{enumerate}
\ep
 \end{proof}
A more precise estimation of $\Gamma ( Q)$ will be needed:
\begin{Proposition}\label{cont1}
There is a positive constant $K$ which depends only on
 $\alpha ,\bar{\alpha} ,\beta ,\delta ,T,(U_s)_{s\in[0,T]},\bar{U}_T$
such that
$$
d(Q|P)\leq K( 1+\Gamma ( Q)).
$$

In particular $\inf\limits_{Q\in Q_{f}}$ $\Gamma (Q) >-\infty. $
\end{Proposition}
 \begin{proof}
From Bayes' formula, we have:
\begin{equation*}
\begin{split}
\E_Q [ \dint_{0}^{T}\delta _sS_s^\delta\frac{f( Z_s^Q)}{Z_s^Q}
ds|\Fc_\tau] &= \frac{1}{Z_\tau^Q}\E_{P}[\dint_0^T\delta _sS_s^\delta f( Z_s^Q)ds|\Fc_\tau]
 \geq -\frac{1}{Z_\tau^Q}T \kappa \parallel \delta \parallel_\infty.
\end{split}
\end{equation*}
In the same way, by using $\exp(-T\parallel\delta\parallel _{\infty})\leq S_{T}^{\delta }\leq 1$, we get:
\begin{equation}\label{inequa6}
\begin{split}
\E_{Q}[S_T^\delta \dfrac{f(Z_T^Q)}{Z_T^Q}|\Fc_\tau]& =\frac{1}{Z_\tau^Q}\E_P[
S_T^\delta f( Z_{T}^{Q})|\mathcal{F}_{\tau}] \\& = \frac{1}{Z_\tau^Q} \E_P\big[
S_T^\delta [f( Z_{T}^{Q})-\kappa+\kappa]|\mathcal{F}_{\tau}\big]
\\& \geq \frac{1}{Z_\tau^Q}(-\kappa + e^{-T \parallel\delta \parallel _{\infty
}}(\kappa +\E_P[ f( Z_{T}^{Q})|\mathcal{F}_{\tau}])
\\& \geq \frac{1}{Z_\tau^Q}(-\kappa + e^{-T \parallel\delta \parallel _{\infty
}}\E_P[ f( Z_{T}^{Q})|\mathcal{F}_{\tau}]).
\end{split}
\end{equation}
We set $R_\tau:= \alpha \int_{\tau}^{T}|U_s|ds+\bar{\alpha}|\bar{U}_T|$ and $R=R_0.$ By using $0 \leq S^\delta \leq 1,$  and Bayes' formula, we have
\begin{equation}\label{inequa7}
\begin{split}
\E_Q[\Uc^\delta_{0,T}|\Fc_\tau]&\geq -\E_Q[R|\Fc_\tau]\\&=-\frac{1}{Z_\tau^Q}\E_P[Z_T^Q R|\Fc_\tau].
\end{split}
\end{equation}
By using (\ref{convex inequality2}) and since $f^*$ is nondecreasing, we obtain:
\begin{equation}\label{inequa7bis}
\begin{split}
\E_P[Z_T^QR|\Fc_\tau]&\leq \frac{1}{\gamma}\E_P[f(Z_T^Q)+f^*(\gamma R)|\Fc_\tau]
\\&\leq\frac{1}{\gamma}\E_P[f(Z_T^Q)|\Fc_\tau]+\frac{1}{\gamma}\E_P[f^*(\gamma R)|\Fc_\tau].
\end{split}
\end{equation}
Thus, plugging inequality (\ref{inequa7bis}) into inequality (\ref{inequa7}) and adding the inequality (\ref{inequa6}) we get:
\begin{equation}\label{ineq7}
\begin{split}
\E_Q[c(.,Q)|\Fc_\tau]&\geq -\beta \frac{1}{Z_\tau^Q}T \kappa \Vert \delta \Vert_\infty +\beta \frac{1}{Z_\tau^Q}\Big(-\kappa + e^{-T \Vert\delta \Vert _{\infty}}\E_P[ f( Z_{T}^{Q})|\mathcal{F}_{\tau}]\Big)\\& -\frac{1}{Z_\tau^Q}\Big(\frac{1}{\gamma}\E_P[f(Z_T^Q)|\mathcal{F}_{\tau}]+\frac{1}{\gamma}\E_P[f^*(\gamma \alpha \int_{0}^{T}|U_s|ds+\gamma \bar{\alpha}|\bar{U}_T| )|\Fc_\tau]\Big).
\end{split}
\end{equation}
By  choosing $ \tau=0$ and taking the expectation under $Q$, we  obtain:
\begin{equation}
\begin{split}\Gamma(Q)\geq &-\beta T \kappa \Vert \delta \Vert_\infty  +\beta [-\kappa + e^{-T \Vert\delta \Vert _{\infty}}\E_P[f(Z_{T}^{Q})]]-\\& \Big(\frac{1}{\gamma}\E_P[f(Z_T^Q)]+\frac{1}{\gamma}\E_P[f^*(\gamma \alpha \int_{0}^{T}|U_s|ds+\gamma \bar{\alpha}|\bar{U}_T| )]\Big)
\\&=-\beta \kappa (T \Vert \delta \Vert_\infty+1)+d(Q|P)(\beta e^{-T \Vert\delta \Vert _{\infty}}-\frac{1}{\gamma})
\\&-\frac{1}{\gamma}\E_P\Big[f^*(\gamma \alpha \int_{0}^{T}|U_s|ds+\gamma \bar{\alpha}|\bar{U}_T|)\Big].
\end{split}
\end{equation}
By choosing $\gamma $ large enough, there exists $\eta>0$  such that $ \beta e^{-T \Vert\delta \Vert _{\infty}}-\frac{1}{\gamma} \geq \eta.$
We set $$K:= \frac{1}{\eta} \max (1,\beta \kappa (T \Vert \delta \Vert_\infty+1)+\frac{1}{\gamma}\E_P[f^*(\gamma \alpha \int_{0}^{T}|U_s|ds+\gamma \bar{\alpha}|\bar{U}_T|)].$$
\\Under the assumptions \textbf{(A1)-(A2)}, $K$ is finite and so the proof of the proposition is achieved.
\ep
 \end{proof}
The following lemma is useful to show the existence of $Q^*$ which realizes the infimum of $Q\mapsto\Gamma(Q)$
\begin{Lemma}\label{lemma2}
For all $ \gamma >0$ and all $A \in \Fc_T$ we have :
\begin{equation}
\E_Q\Big[|\Uc^\delta_{0,T}|\textbf{1}_A\Big] \leq \frac{1}{\gamma}(d(Q|P)+\kappa)+\frac{1}{\gamma}\E_P\Big[f^*(\gamma \alpha \int_{0}^{T}|U_s|ds+\gamma \bar{\alpha}|\bar{U}_T|)\textbf{1}_A\Big].
\end{equation}
\end{Lemma}
 \begin{proof}
From the definition of $\Uc^\delta_{0,T}$ and using inequality (\ref{convex inequality1}), we have
\begin{equation*}
\begin{split}
Z_T^Q|\Uc^\delta_{0,T}|\textbf{1}_A &\leq Z_T^Q \Big(\alpha \int_{0}^{T}|U_s|ds+\gamma \bar{\alpha}|\bar{U}|\Big)\textbf{1}_A
\\& \leq \frac{1}{\gamma}\Big(f(Z_T^Q)+f^*(\gamma \alpha \int_{0}^{T}|U_s|ds+\gamma \bar{\alpha}|\bar{U}|)\Big)\textbf{1}_A.
\end{split}
\end{equation*}
Using assumption \textbf{(H2)}, we obtain
\begin{equation*}
\begin{split}
Z_T^Q|\Uc^\delta_{0,T}|\textbf{1}_A & \leq \frac{1}{\gamma}\Big(f(Z_T^Q)+\kappa+f^*(\gamma \alpha \int_{0}^{T}|U_s|ds+\gamma \bar{\alpha}|\bar{U}|)\Big)\textbf{1}_A
\\& = \frac{1}{\gamma}[f(Z_T^Q)+\kappa]+\frac{1}{\gamma}[f^*(\gamma \alpha \int_{0}^{T}|U_s|ds+\gamma \bar{\alpha}|\bar{U}|)]\textbf{1}_A.
\end{split}
\end{equation*}
The result follows by taking the expectation under $P$.\\
\ep
\end{proof}
The next result is related to the existence of unique probability measure solution of optimization problem (\ref{min1})
\begin{Theorem}\label{exisun} Under assumptions \textbf{(A1)-(A2)} and \textbf{(H3)}, there is  a unique $Q^*\in \Qc_f$ which minimizes
$Q \mapsto \Gamma(Q)$ over all $Q \in \Qc_f$ .
\end{Theorem}
\begin{proof}
 \begin{enumerate}
 \item $Q \mapsto \Gamma(Q)$ is strictly convex; hence $Q^*$ must be unique if it exists.
\item Let $(Q^n)_{n\in \N}$ be a minimizing sequence in $\Qc_f$ i.e.
$$\searrow \lim\limits_{n \rightarrow +\infty}\Gamma(Q^n)=\inf\limits_{Q\in\Qc_f}\Gamma(Q)>-\infty, $$
and we denote by  $Z^n = Z^{Q^n}$ the corresponding density processes.
\\ Since each $Z_T^n\geq 0$, it follows from Koml$\acute{o}$s' theorem that there exists a sequence $(\bar{Z}_T^n )_{n\in \N}$ with $\bar{Z}_T^n \in  conv(Z^n_T ,Z^{n+1}_T , ...)$ for each $n\in\N$ and such that  $(\bar{Z}^n_T )$ converges $P$-a.s. to some random variable $\bar{Z}^{\infty}_T$ which is nonnegative but may take the value $+\infty.$
Because  $\Qc_f$ is convex, each $\bar{Z}_T^n$ is again associated with some $\bar{Q}^n \in \Qc_f.$
  We claim that this also holds for $\bar{Z}^{\infty}_T$ , i.e., that $d\bar{Q}^{\infty} :=\bar{Z}^{\infty}_TdP$ defines a probability measure $\bar{Q}^{\infty}\in \Qc_f.$ To see this, note first that we have
\begin{equation}\label{ineq1}
\Gamma(\bar{Q}^n) \leq \sup\limits_{m\geq n}\Gamma(Q^m)= \Gamma(Q^n)\leq \Gamma(Q^1),
\end{equation}
  because $Q \mapsto \Gamma(Q)$ is convex and $n \mapsto \Gamma(Q^n)$ is decreasing. Hence Proposition $\ref{cont1}$ yields
\begin{equation}\label{ineq2}
\begin{split}
\sup\limits_{n\in \N}\E_P [f (\bar{Z}^n) ] &= \sup\limits_{n\in \N}d(\bar{Q}^n|P) \leq K(1 + \sup\limits_{n\in \N}\Gamma(\bar{Q}^n))\\& \leq K(1 + \sup\limits_{n\in \N}\Gamma({Q}^n))\leq K(1 +\Gamma({Q}^1))<+\infty.
\end{split}
\end{equation}
From assumption \textbf{(H3)} and using de la Vall\'ee-Poussin's  criterion, we obtain the $P$-uniformly integrability of $(\bar{Z}^n_T )_{n\in\N}$   and therefore $(\bar{Z}^n_T )_{n\in\N}$ converges in $L^1(P).$ This implies that $\E_P[\bar{Z}^{\infty}_T ]=\lim\limits_{n\rightarrow +\infty}E_P[\bar{Z}^n_T]=1$ and so $\bar{Q}^{\infty}$ is  a probability measure which is absolute continuous with respect to  $P$ on $\Fc_T$ . Because $f$ is bounded from below by $\kappa$, Fatou's lemma and inequality(\ref{ineq2}) yield
\begin{equation}\label{c3}
 d(\bar{Q}^\infty|P) = \E_P[f(\bar{Z}^\infty_T)] \leq \liminf\limits_{n\rightarrow +\infty} \E_P[f(\bar{Z}^n_T)] < +\infty.
\end{equation}
Finally, we also have $\bar{Q}^\infty = P $ on $\Fc_0.$ In fact, $(\bar{Z}^n_T)$ converges  to $\bar{Z}^\infty_T$ strongly in $L^1(P)$, hence also weakly in $L^1(P)$ and so we have for every  $A \in \Fc_0:$
$$\bar{Q}^\infty[A] = \E_P[\bar{Z}^\infty_T \textbf{1}_A] =\lim\limits_{n \rightarrow +\infty} \E_P [Z^n_T \textbf{1}_A] =\lim\limits_{n\rightarrow +\infty}\bar{Q}^n[A]=P[A].$$
The last equality holds since $\bar{Q}^n(A)=P(A)$ for all $n\in \mathbb{N}$ and $A \in \Fc_0$. This shows that $\bar{Q}^{\infty}\in\Qc_f $.
\item We now want to show that $Q^* := \bar{Q}^\infty$ attains the infimum of
$Q \mapsto\Gamma(Q)$ on $\Qc_f.$
\\Let $\bar{Z}^\infty$ be the density process of  $\bar{Q}^\infty$ with respect to $P.$
Because we know that $(\bar{Z}^n_T)$ converges to $\bar{Z}^\infty$  in  $L^1(P)$, Doob's maximal inequality
$$P[\sup\limits_{0\leq t \leq T} \mid \bar{Z}^\infty_t-\bar{Z}^n_t \mid\geq \epsilon]\leq \frac{1}{\epsilon} \E_P[\mid \bar{Z}^\infty_T-\bar{Z}^n_T \mid]$$
implies that  $(\sup\limits_{0\leq t \leq T} \mid \bar{Z}^\infty_t-\bar{Z}^n_t \mid)_{n\in\N}$ converges to $0$ in $P$-probability.
\\By passing to a subsequence that we still denote by $(\bar {Z}^n)_{n\in\N}$, we may thus assume that  the sequence $(\bar {Z}^n)$ converges to $\bar{Z}^\infty$ uniformly in $t$ with $P$-probability 1. This implies that the sequence
$ (Z^n_T c( .,\bar{Q}^n) )$  converges  to  $ \bar{Z}_T^\infty c(.,\bar{Q}^\infty) P$-a.s.
and in more detail with $$\bar{Y}_1^n:=\bar {Z}^n_T \Uc_{0,T}^\delta , \bar{Y}_2^n:=\beta(\dint_0^T\delta_sS_s^\delta f(\bar{Z}_s^n)ds+S_T^\delta f(\bar{Z}_T^n))=\beta \Rc_{0,T}^\delta(\bar{Q}^n)$$
for $n\in\N \cup \{+\infty\}$ that
$$\lim\limits_{ n\rightarrow +\infty} \bar{Y}^n_i =  \bar{Y}^\infty_i P-a.s. \;\;\textrm{for}\;\; i = 1,2.$$
Since $\bar{Y}^n_2$ is bounded from below, uniformly in $n$ and $\omega$, Fatou's lemma yields
\begin{equation}\label{ineq4}
\E_P[\bar{Y}^\infty_2]\leq \liminf\limits_{n\rightarrow \infty}\E_P[\bar{Y}^n_2].
\end{equation}
We prove below that we have
\begin{equation}\label{ineq5}
\E_P[\bar{Y}^\infty_1]\leq \liminf\limits_{n\rightarrow \infty}\E_P[\bar{Y}^n_1].
\end{equation}
Plugging (\ref{ineq4}) and (\ref{ineq5}) into (\ref{ineq1}), we obtain
$$\Gamma(\bar{Q}^\infty) =\E_P [ \bar{Y}^\infty_1+\bar{Y}^\infty_2 ] \leq \liminf\limits_{n\rightarrow\infty}\Gamma(\bar{Q}^n)
\leq \liminf\limits_{n\rightarrow\infty}\Gamma(Q^n)\leq \inf\limits_{Q\in \Qc_f}\Gamma(Q)$$
which proves that $\bar{Q}^\infty $ is indeed optimal.
\\It now remains to show that $\E_P[\bar{Y}^\infty_1]\leq \liminf\limits_{n\rightarrow \infty}\E_P[\bar{Y}^n_1].$
\\We set for $m\in \N;$ $\tilde{R}_m:=\Uc_{0,T}^\delta \textbf{1}_{\{\Uc_{0,T}^\delta\geq -m\}}.$
Thus for all $n\in \N\cup \{+\infty\}$; $$\bar{Y}^n_1=\bar{Z}_T^n\Uc_{0,T}^\delta=\bar{Z}_T^n\tilde{R}_m+\bar{Z}_T^n\Uc_{0,T}^\delta \textbf{1}_{\{\Uc_{0,T}^\delta< -m\}}.$$
Since $\tilde{R}_m \geq -m$ and $\E_P[\bar{Z}_T^n]=1$,  Fatou's lemma yields :
$$\E_P[\bar{Z}_T^\infty\Uc_{0,T}^\delta]\leq \liminf\limits_{n \rightarrow \infty} \E_P[\bar{Z}_T^n\Uc_{0,T}^\delta].$$
Hence
\begin{equation*}
\begin{split}
\E_P[\bar{Y}^\infty_1]&=\E_P[\bar{Y}^\infty_1 \textbf{1}_{\{\Uc_{0,T}^\delta\geq -m\}}]+\E_P[\bar{Y}^\infty_1 \textbf{1}_{\{\Uc_{0,T}^\delta <-m\}}]
\\&\leq \liminf\limits_{n \rightarrow \infty} \E_P[\bar{Z}_T^n\tilde{R}_m]+\E_P[\bar{Z}_T^\infty\Uc_{0,T}^\delta \textbf{1}_{\{\Uc_{0,T}^\delta <-m\}}]
\\&\leq \liminf\limits_{n \rightarrow \infty} \E_P[\bar{Y}_1^n]+2 \sup\limits_{n\in \N\cup \{\infty\}}\E_P[\bar{Z}_T^n|\Uc_{0,T}^\delta | \textbf{1}_{\{\Uc_{0,T}^\delta <-m\}}].
\end{split}
\end{equation*}
It remains to show that $$\lim\limits_{m\rightarrow +\infty}\sup\limits_{n\in \N\cup \{\infty\}}\E_P[\bar{Z}_T^n|\Uc_{0,T}^\delta| \textbf{1}_{\{\Uc_{0,T}^\delta <-m\}}]=0.$$
However, Lemma \ref{lemma2} and Proposition \ref{cont1} give for any $n\in \N\cup \{\infty\}$:
\begin{equation*}
\begin{split}
&\E_P\Big[\bar{Z}_T^n|\Uc_{0,T}^\delta| \textbf{1}_{\{\Uc_{0,T}^\delta <-m\}}\Big]=\E_{\bar{Q}^n}\Big[|\Uc_{0,T}^\delta| \textbf{1}_{\{\Uc_{0,T}^\delta <-m\}}\Big]
\\&\leq \frac{1}{\gamma}(d(\bar{Q}^n|P)+\kappa)+\frac{1}{\gamma}\E_P\Big[\textbf{1}_{\{\Uc_{0,T}^\delta <-m\}}f^*(\gamma \alpha \int_{0}^{T}|U_s|ds+\gamma \bar{\alpha}|\bar{U}_T| )\Big]
\\& \leq \frac{1}{\gamma}\Big(K\big(1+\Gamma(\bar{Q}^n)\big)+\kappa\Big)+\E_P\Big[\textbf{1}_{\{\Uc_{0,T}^\delta <-m\}}f^*(\gamma \alpha \int_{0}^{T}|U_s|ds+\gamma \bar{\alpha}|\bar{U}_T| )\Big].
\end{split}
\end{equation*}
Using inequality (\ref{ineq1}), we obtain for all $\gamma >0$
\begin{equation*}
\begin{split}
\sup\limits_{n\in \N\cup \{\infty\}}\E_P[\bar{Z}_T^n|\Uc_{0,T}^\delta| \textbf{1}_{\{\Uc_{0,T}^\delta <-m\}}]& \leq \frac{1}{\gamma}\Big(K\big(1+\Gamma(Q^1)\big)+\kappa\Big)
\\&+\frac{1}{\gamma}\E_P[\textbf{1}_{\{\Uc_{0,T}^\delta <-m\}}f^*(\gamma \alpha \int_{0}^{T}|U_s|ds+\gamma \alpha'|U_T^{\prime }|)]).
\end{split}
\end{equation*}
\\Thanks to the dominated convergence theorem, the integrability assumption (\textbf{A2}) and since $f^*$ is nonnegative function, we have $$\lim\limits_{m\rightarrow \infty}\E_P[\textbf{1}_{\{\Uc_{0,T}^\delta <-m\}}f^*(\gamma \alpha \int_{0}^{T}|U_s|ds+\gamma \bar{\alpha}|\bar{U}_T|)]=0.$$
Therefore, for all $\gamma >0$
$$\lim\limits_{m\rightarrow +\infty}\sup\limits_{n\in \N\cup \{\infty\}}\E_P[\bar{Z}_T^n|\Uc_{0,T}^\delta| \textbf{1}_{\{\Uc_{0,T}^\delta <-m\}}] \leq \frac{1}{\gamma}(K(1+\Gamma(Q^1))+\k).$$
\\Finally by sending $\gamma$ to $+\infty$, the desired result is obtained.\ep
\end{enumerate}
\end{proof}
Our next aim  is to prove that the minimal measure $Q^*$ is equivalent to $P.$ We use  an adaptation of an argument given by Frittelli \cite{Fri00}, which is the object of the following lemma.
\begin{Lemma}\label{lemme}
We assume \textbf{(H2)}.
Let $Q^0$ and $Q^1$ two elements in $\Qc_f$ with respective densities $Z^0$ and $Z^1.$ Then
$$ \sup\limits_{0\leq t\leq T}\E_P\Big[\Big(f'(Z_t^0)(Z_t^1-Z_t^0)\Big)^+\Big]\leq d(Q^1|P)+\kappa.$$
\end{Lemma}
\begin{proof}
Set $Z^x=xZ^1+(1-x)Z^0$ and for $x\in(0,1]$ and fixed  $t\in \R,$
\begin{equation}\label{star}H(x,t):=\frac{1}{x}(f(Z_t^x)-f(Z_t^0)).\end{equation}
Since $f$ is strictly convex, the function $x\mapsto H(x,t)$ is nondecreasing  and consequently
\begin{equation*}
\begin{split}
H(1,t)&\geq\lim\limits_{x\searrow 0}\frac{1}{x}(f(Z_t^x)-f(Z_t^0))=\frac{d}{dx}f(Z_t^x)\mid_{x=0}\\&=f'(Z_t^0)(Z_t^1-Z_t^0).
\end{split}
\end{equation*}
From assumption \textbf{(H2)}, we have:
\begin{equation}\label{equa1}
\begin{split}
f'(Z_t^0)(Z_t^1-Z_t^0)\leq H(1,t)&=f(Z_t^1)-f(Z_t^0)\\&\leq f(Z_t^1)+\kappa.
\end{split}
\end{equation}
Since $f(Z_t^1)+\k\geq 0,$ then $(f'(Z_t^0)(Z_t^1-Z_t^0))^+ \leq f(Z_t^1)+\kappa.$
Replacing in the inequality (\ref{c1}) $Z^Q$ by $Z^1$, we obtain:
$$\E_P[f(Z_t^1)]\leq \E_P[f(Z_T^1)]=d(Q^1|P).$$
Taking the expectation under $P$ in equation (\ref{equa1}) the desired result is obtained.
\ep
\end{proof}
\begin{Theorem}  Under the assumptions \textbf{\textrm{(H2),(H4)}} and  \textbf{\textrm{(A1)-(A2)}}, the optimal probability measure $Q^*$ is equivalent to $P.$
\end{Theorem}
\begin{proof}
\\ 1) As in the proof of Lemma \ref{lemme}, we take $Q^0,Q^1 \in \Qc_f$ , we set $Q^x := xQ^1 +(1-x)Q^0$ for
$x\in (0; 1]$ and we denote by $Z^x$ the density process of $Q^x$ with respect to $P$. Then, get
\begin{equation*}
\begin{split}
\frac{1}{x}(\Gamma(Q^x)-\Gamma(Q^0))&= \E_P [(Z^1_T-Z^0_T) \Uc^\delta_{0,T} ]
\\&+\frac{1}{x}\beta\E_P[\dint_0^T \delta_sS_s^\delta (f(Z^x_s )-f(Z^0_s ))ds + S_T^\delta (f(Z^x_T )-f(Z^0_T))]
\\&= \E_P[(Z^1_T-Z^0_T) \Uc^\delta_{0,T} ]
\\&+\beta\E_P[\dint_0^T \delta_sS_s^\delta H(x,s)ds + S_T^\delta H(x,T)].
\end{split}
\end{equation*}
Since  $x\mapsto H(x; s)$ is nondecreasing and using assumption \textbf{(H2)}, we have
$$H(x,s) \leq  H(1,s) = f(Z^1_s )-f(Z^0_s ) \leq f(Z^1_s )+\kappa,$$
where the right hand of the last inequality is integrable. Hence  monotone convergence Theorem  can be used to deduce that
\begin{equation}\label{derivegamma}
\begin{split}
\frac{d}{dx} \Gamma(Q^x)\mid_{x=0}&= \E_P[(Z^1_T-Z^0_T) \Uc^\delta_{0,T}]+\beta \E_P[\dint_0^T \delta_sS_s^\delta f'(Z_s^0)(Z_s^1-Z_s^0)ds \\&+ S_T^\delta f'(Z_T^0)(Z_T^1-Z_T^0)]
\\&:= \E_P[Y_1] + \E_P[Y_2].
\end{split}
\end{equation}
Under assumptions (\textbf{A1})-(\textbf{A2}) and from inequality (\ref{convex inequality1}), we have $Y_1\in L^1(P)$. As in the proof of Lemma \ref{lemme}, using the nondecreasing property of the function  $x \mapsto H(x,s)$ and assumption (\textbf{H2}), we obtain
$$Y_2\leq\int_0^T\delta_sS^\delta_sH(1,s) ds + S^\delta_T H(1,T)\leq \int_0^T\delta_sS^\delta_s(f(Z_s^1)+\k)ds + S^\delta_T (f(Z_T^1)+\k)$$
which is $P$-integrable because $Q^1\in \Qc_f$ . From Lemma \ref{lemme} we deduce that $Y_2^+ \in L^1(P)$ and so the right-hand side of (\ref{derivegamma}) is well-defined in $[-\infty,+\infty)$.
\\2) Now take $Q^0 = Q^*$ and any $Q^1\in \Qc_f$ which is equivalent to $P$ this is possible since $\Qc_f$ contains $P$. The optimality of $Q^*$ yields $\Gamma(Q^x)-\Gamma(Q^*)\geq0$ for all $x \in (0; 1]$, hence also
\begin{equation}\label{d1}
\frac{d}{dx} \Gamma(Q^x)\mid_{x=0}\geq 0.
\end{equation}
Therefore the right-hand side of (\ref{derivegamma}) is nonnegative.
Which implies that $E_P[Y_2^-]\leq E_P[Y_1]+E_P[Y_2^+].$ The right hand side of the last inequality is finite since $Y_1\in L^1(P)$ and $Y_2^+\in L^1(P).$ This shows that $Y_2$ must be in $L^1(P)$.
This makes it possible to rearrange terms and rewrite (\ref{d1}) by using (\ref{derivegamma}) as
\begin{equation}\label{ineq6}
\beta \E_P[\dint_0^T \delta_sS_s^\delta f'(Z_s^*)(Z_s^1-Z_s^*)ds+ S_T^\delta f'(Z_T^*)(Z_s^1-Z_T^*)]\geq -\E_P[(Z^1_T-Z^*_T) \Uc^\delta_{0,T}].
\end{equation}
But the right-hand side of (\ref{ineq6}) is strictly greater than $ -\infty.$ So, if the probability measure $Q^*$ is not equivalent to $P$, then the set $A :=\{Z_T^*=0\}$ satisfies  $P[A] > 0$. Since $Q^1\approx P$, we have $Z_T^1>0.$ From assumption (\textbf{H4}), we have  $(f'(Z_T^*)(Z_T^1-Z_T^*))^-=+\infty$ on $A$. It follows that $\E_P[(f'(Z_T^*)(Z_T^1-Z_T^*))^-]=\infty.$
From Lemma \ref{lemme}, we know  that $[(f'(Z_T^*)(Z_T^1-Z_T^*))^+] \in L^1(P)$,
then we obtain $\E_P[f'(Z_T^*)(Z_T^1-Z_T^*)]=-\infty.$ This gives a contradiction to (\ref{ineq6}). Therefore $Q^*\thickapprox P.$
\ep
\end{proof}
\subsection{Bellman optimality principle}
In this section we establish the martingale optimality principle which is a direct consequence of Theorems 1.15 , 1.17 and 1.21 in El Karoui \cite{ELK81}.
For this reason, some notations are introduced. Let $\Sc$ denote the set of all $\Fc$-stopping times $\tau$ with
values in $[0,T]$ and $\Dc$ the space of all density processes $Z^Q$ with $Q\in \Qc_f$ . We define
$$\Dc(Q,\tau):=\{Z^{Q'}\in \Dc; Q=Q'\; \textrm{on} \;\Fc_\tau \},$$
$$\Gamma(\tau,Q):= \E_Q[c(.,Q)|\Fc_\tau],$$
and the minimal conditional cost at time $\tau$ ,
$$J(\tau,Q) := Q\;\;\textrm{-}\underset{Q'\in \Dc(Q,\tau)}{\essinf}\Gamma(\tau,Q').$$
Then the optimization problem (\ref{min1}) can be reformulated to
\begin{equation}\label{remin} \textrm{find} \inf\limits_{Q\in \Qc_f}\Gamma(Q) = \inf\limits_{Q\in \Qc_f}\E_Q[c(.,Q)] = \E_P [J(0;Q)], \end{equation}
by using the dynamic programming principle and the fact that $Q=P$ on $\Fc_0$ for every $Q\in \Qc_f.$
 \\ In the following, the  Bellman martingale optimality principle  is given. \begin{Proposition}
 1. The family  $\{J(\tau, Q)|\tau \in \mathcal{S},Q \in Q_f\}$ is a submartingale system i.e.
 $$\textrm{for all} \,\, (\tau, \tau')\in \mathcal{S}^{2} s.t \;\; \tau\geq \tau'; E[J(\tau, Q)|\Fc_{\tau'}]\geq J(\tau', Q) $$
 \\2. $Q^*\in Q_f$ is optimal $\Leftrightarrow$ $\{J(\tau, Q^*)|\tau \in \mathcal{S}\}$ is a martingale system i.e.
  $$\textrm{for all} \,\, (\tau, \tau')\in \mathcal{S}^{2} s.t \;\; \tau\geq \tau'; E[J(\tau, Q^{*})|\Fc_{\tau'}]= J(\tau', Q^{*}) $$
 \\3. For all $Q \in Q_f$ there is an adapted RCLL process  $J^Q=(J_t^Q)_{0\leq t \leq T}$ which is a right closed $Q$-submartingale such that :  $J_{\tau}^Q=J(\tau, Q)\;\; Q$-a.s for each stopping time $\tau.$\\
\end{Proposition}

The proof is  given in the Appendix. Moreover, we should to apply Theorems 1.15, 1.17, 1.21 in El Karoui \cite{ELK81}. These results require that:
\begin{enumerate}
\item[1.]\label{condition1} $c\geq 0$ or $\displaystyle \inf_{Q'\in \Dc(Q,t)}E_{Q'}[|c(\cdot,\ Q')|]<\infty$ for all $\tau\in \Sc$ and $Q\in \mathcal{Q}_{f}$,
\item[2.] The space $\mathcal{D}$ is compatible i.e.
\\For $Z^Q\in \Dc, \tau \in \Sc$ and $Z^{Q'}\in \Dc(Q,\tau)$,  we have $Q|_{\Fc_\tau}=Q'|_{\Fc_\tau}.$
\item[3.] The space $\mathcal{D}$  stable under bifurcation i.e.
\\For all  $Z^Q\in \Dc,\tau \in \Sc , A\in \Fc_\tau$ and $Z^{Q'}\in \Dc(Q,\tau)$, we have
$$Z^Q|\tau_A|Z^{Q'}:=Z^{Q'}\mathbf{1}_A+Z^Q\mathbf{1}_{A^c} \in \Dc(Q,\tau).$$
\item[4.] The cost functional is coherent i.e.
\\For all $Z^Q\in \Dc$ and $Z^{Q'}\in \Dc$, we have $c(w,Q)=c(w,Q')$ on the set  $\{w,Z_T^Q(w)=Z_T^{Q'}(w)\}$ $Q-a.s$ and $Q'-a.s.$
\end{enumerate}

\begin{Remark}
In the proof of the Bellman Optimality principle, condition (\ref{condition1}) ensures that $J(\tau,Q)\in L^1(Q)$ for each $\tau\in \Sc.$  In this case we prove such a result directly (see Lemma \ref{lemma3}).
\end{Remark}

\section{Class of Consistent time penalty}
In this section, we assume that  filtration $(\Fc_t)_{0\leq t \leq T}$ is generated by a d-dimensional Brownian motion $W.$ Then, for every measure $Q\ll P$ on $ \Fc_T $, there is a  progressively measurable process $ (\eta_t)_{0 \leq  t \leq T} $ which takes values in $\mathbb{R}^d$ such that $ \int_0^T \| \eta_t \|^2 dt<+ \infty, \; P.a.s $ and the density  process of  $Q$ with respect to $ P $ is an RCLL local martingale  $Z^\eta=(Z^\eta_t)_{0 \leq t \leq T} $ given by:
\begin{equation}\label{representation}Z_t^\eta = \Ec(\int_0^t\eta_udW_u) \,\; Q.p.s, \forall t \in [0,T]. \end{equation}
where $\Ec(M)_t = \exp (M_t -\frac{1}{2}\langle M \rangle_t)$ denotes the stochastic exponential of a continuous local martingale $M$.
$Q^\eta$ denotes the measure which admits $Z^\eta$ as density with respect to the reference probability measure $P.$
We introduce a consistent time penalty given by:
$$ \gamma_t(Q^\eta) = E_{Q^\eta} [\int_t^T h (\eta_s) ds | \Fc_t] $$
where $ h:  \R^d \rightarrow [0, + \infty] $ is a convex function, proper and lower semi-continuous function such that $ h (0)= 0 .$
We also assume that there are two positive constants $ \kappa_1 $ and $ \kappa_2 $ satisfying:
$$ h(x) \geq \kappa_1 \| x\|^2 - \kappa_2.$$
The penalty term is defined by
\begin{equation}\label{penalty definition ctc} \Rc_{t, T}^\delta(Q^\eta)= \dint_t^ T \delta_s \frac{S_s^\delta}{S_t ^ \delta} (\int_t^s h(\eta_u) du) ds + \frac{S_T^\delta} {S_t^\delta} \int_t^T h (\eta_u) du,  \; \forall  \, 0 \leq t \leq T \end{equation}
for $ Q \ll P $ on $\Fc_T .$
As in the case of $f$-divergence penalty,  the following optimization problem has to be solved:
\begin{equation} \label{min}
\textrm{minimize the functional} \, Q^\eta \mapsto \Gamma(Q^\eta):=\E_{Q^{\eta}}[c (.,Q^\eta)]
\end{equation}
over an appropriate class of probability measures $ Q^\eta \ll P. $
\begin{Definition}
For each probability measure $Q^\eta $ on $ (\Omega, \Fc) $, the penalty function is defined:
\begin{equation*} \label{penalty}
\gamma_t(Q^\eta):=\left\{\begin{array}{cc}
\E_{Q^\eta} [\dint_t ^T h(\eta_s)ds | \Fc_t] & \textrm {if}\, \, Q^\eta \ll P \,\, \textrm {on} \, \Fc_T \\
+ \infty & \textrm{otherwise}.
\end{array} \right.
\end{equation*}
We note $\Qc_f^c $ the space of all probability measures $ Q^\eta $ on $(\Omega, \Fc) $ such that $Q^\eta \ll P$ on $\Fc_T$ and $\gamma_0 (Q^\eta)< +\infty$ and $ \Qc^{c,e}_f:=\{Q^\eta \in \Qc_f^{c} | Q \approx P \, \textrm {on} \, \Fc_T \} $.
\end{Definition}
\begin{Remark}\label{rem2}
\begin{enumerate}
\item[1-] We note that $\Qc_f^{c,e}$ is a  non empty set because $ P \in \Qc_f^{c,e}.$
\item[2-]
The particular case of $ h(x)= \frac{1}{2}|x|^2 $ corresponds to the entropic  penalty. Indeed
\begin{equation*}
\begin{split}
H(Q^\eta|P)&= \E_{Q^\eta} [\log(\frac{dQ^\eta}{dP})]
\\ &=\E_{Q^\eta} [\int_0^T\eta_u dW_u-\frac{1}{2}\int_0^T|\eta_u|^2 du]
\end{split}
\end{equation*}
Since $(\int_0^.\eta_u dW_u)$ is a local martingale under $P$, then by the Girsanov theorem $(\int_0^.\eta_u dW_u)-\int_0^.|\eta_u|^2 du$ is a local martingale under $Q^\eta$ and so
\begin{equation*}
\begin{split}
H(Q^\eta|P) &=\E_{Q^\eta} [\int_0^T\eta_u dW_u-\int_0^T|\eta_u|^2 du+\frac{1}{2}\int_0^T|\eta_u|^2 du]
\\&=\E_{Q^\eta} [\frac{1}{2}\int_0^T|\eta_u|^2 du]=\gamma_0(Q^\eta).
\end{split}
\end{equation*} \ep
\item[3-] For a  general function $h$ we have for all $Q^\eta \in \Qc_f^c$,
\begin{equation}\label{estimation entropy} H(Q^\eta | P) \leq\frac{1}{2\kappa_1}\gamma_0 (Q^\eta) + \frac{T\kappa_2}{2\kappa_1}.
\end{equation}
Indeed:
\begin{equation} \label{ineqalityentropy}
\begin{split}
H(Q^\eta|P) = \E_{Q^\eta}[\frac{1}{2} \dint_0^T |\eta_s|^2 ds] & \leq \E_P [\frac {1}{2\kappa_1}(\dint_0^T (h(|\eta_s|) + \kappa_2) ds)]
\\ & \leq \E_{Q^\eta} [\frac{1}{2\kappa_1}(\dint_0^T (h (|\eta_s|)ds) + \frac{T\kappa_2}{2 \kappa_1}]
\\ & = \frac{1}{2\kappa_1}\gamma_0 (Q^\eta) + \frac{T\kappa_2}{2\kappa_1}
\end{split}
\end{equation}
\\In particular $H(Q^\eta|P)$ is finite for all $Q^\eta \in \Qc_f^c.$\ep
\end{enumerate}
\end{Remark}
The well-posdness of the problem (\ref{min}) is guaranteed by the integrability condition of $c(.,Q^\eta)$ under $Q^\eta.$ Since $\gamma_0(Q^\eta)< +\infty$ and $h$ takes values on $[0,+\infty]$ together with assumption \textbf{(A1)}, we have for all $Q^\eta \in \Qc_f^c$; $\E_{Q^\eta}[\Rc_{0,T}(Q^\eta)]<+\infty.$
It remains to have $\E_P[Z^{\eta}|\Uc_{0,T}|]<+\infty.$ We apply the inequality (\ref{convex inequality1}) with $f(x)=x\log x$, we obtain
\begin{equation}\label{Aprimejustification}
\E_P[Z^\eta|\Uc^\delta_{0,T}|]\leq\E_P[Z\eta\log Z^\eta +e^{|\Uc^\delta_{0,T}|-1}]
\end{equation}
From the inequality (\ref{estimation entropy}), we have $H(Q^\eta|P)=\E_P[Z^\eta\log Z^\eta]<+\infty.$ Then the right hand side of (\ref{Aprimejustification}) is finite if we  replace the assumption (\textbf{A2}) by:
\\
\\ \textbf{(\textrm{A'2})}    The cost process $U$ belongs to $D_1 ^{\exp}$ and the terminal target $\bar{U}$ is in $L^{\exp}$.
\begin{Remark}\label{integrability of U}
Under Assumption \textbf{(\textrm{A'2})} , we have
\begin{equation} \label{integrability}
\lambda \int_ {0}^{T} |U_s|ds + \mu |\bar{U}_T|\in L^{\exp},\;\; \textrm{for all}\;\;  (\lambda, \mu) \in \mathbb{R}^2_{+}.
\end{equation}
Indeed, since $ x \mapsto \exp(x) $ is convex , we have
\begin{equation*}
\begin{split}
&\mathbb{E}_{P} [\exp (\lambda \int_{0}^{T} |U_s|ds + \mu |\bar{U}_T|)]
\\ & =\mathbb{E}_{P} [\exp (\frac{1}{2} \times 2 \lambda \int_{0}^{T} |U_s|ds + \frac{1}{2}\times 2 \mu |\bar{U}_T|)]
\\ & \leq \mathbb{E}_{P} [\frac{1}{2} \exp (2 \lambda \int_{0}^{T} |U_s| ds) + \frac{1}{2}\exp (2\mu |\bar{U}_T|)]
\\ & = \frac{1}{2} \mathbb{E}_{P} [\exp (2 \lambda \int_{0}^{T} | U_s|ds)] + \frac{1}{2} \mathbb{E}_{P} [\exp (2 \mu |\bar{U}_T|)],
\end{split}
\end{equation*}
which is finite by assumption \textbf{(\textrm{A'2})} .
\end{Remark}
\subsection{Existence of an optimal model}
The main result of this section is to prove the existence of a unique probability $ Q^{\eta*} $ that minimizes the functional $Q^\eta\mapsto \Gamma(Q^\eta)$ in all probability $ Q^\eta \in \Qc_f^c $.
We begin this section by giving some estimates for $\Gamma(Q^\eta)$ for all $Q^\eta \in \Qc_f^c.$
\begin{Proposition}\label{prop3.1}
Under assumptions \textbf{(A1)-(A'2)}, we have for all $Q^\eta \in \Qc_f^c:$
\begin{enumerate}
\item $c(.,Q^\eta)\in L^1(Q^\eta).$
\item $\Gamma (Q^\eta) \leq C(1+\gamma_0 (Q^\eta))$ for some positive constant $C$ which depends only on $\alpha ,\bar{\alpha} ,\beta ,\delta ,T,(U_s)_s\in[0,T])$ and $\bar{U}_T.$
\end{enumerate}
In particular $ \Gamma (Q^\eta) $ is well defined and finite for all $ Q^\eta\in \Qc_f^c.$
\end{Proposition}
\begin{proof}
\begin{enumerate}
\item Similar arguments as used in Proposition \ref{prop1} insure that  $R$ belongs to $L^1(Q^\eta).$ In addition, using assumption (\textbf{A1}), we get
\begin{equation}
\begin{split}
&\vert\dint_0^T\delta_s S_s^\delta (\int_0^sh(\eta_u)du)ds +  S_T^\delta\int_0^T h(\eta_s)ds \vert
\\&\leq \dint_0^T\Vert \delta \Vert_\infty \Big(\int_0^T h(\eta_u)du\Big) ds +\int_0^T h(\eta_s)ds
\\&\leq \Big( \Vert \delta \Vert_\infty T+1\Big)\int_0^T h(\eta_s)ds \in L^1(Q^\eta).
\end{split}
\end{equation}
\item From inequality (\ref{convex inequality1}) with $f(x)=x\log x$, we have:
\begin{equation*}
\begin{split}
\Gamma(Q^\eta) & \leq \E_P[Z^\eta_TR] + \beta \E_{Q^{\eta}} [\dint_0^T\delta_s S_s^\delta (\int_0^sh(\eta_u)du)ds + S_T^\delta\int_0^T h(\eta_u)du]
\\ & \leq \E_P[Z^\eta_T\log Z^{\eta}_T + e^{-1} e^R] + \beta(\parallel\delta\parallel_\infty T+1) \E_Q^{\eta} [\int_0^T h(\eta_u)]
\\ & \leq H(Q^\eta|P) + e^{-1} \E_P[e^R] + \beta (\parallel\delta\parallel_\infty T+1) \gamma_0 (Q^\eta).
\end{split}
\end{equation*}
Inequality (\ref{estimation entropy}) and  assumption (\textbf{A'2}) give the following $$\Gamma(Q^\eta)\leq (\frac{1}{2\kappa_1} + \beta (\parallel\delta\parallel_\infty T+1)) \gamma_0 (Q^\eta) + e^{-1}\E_P[e^R] + \frac{T\kappa_2}{2\kappa_1}.$$
The desired result follows by taking $ C: = \max (e^{-1} E_P [e^R] + \dfrac {t \kappa_2} {2 \kappa_1}, \dfrac{1}{2 \kappa_1} + \beta (\parallel\delta\parallel_\infty T+1)) $
which is finite.\ep
\end{enumerate}
\end{proof}
The following proposition gives a lower bound for our criterion $\Gamma(Q^\eta)$ for all $ Q^\eta\in \Qc_f^c.$
\begin{Proposition} \label{cont}
Under the assumptions \textbf{(A1)-(A'2)}, there exists a positive constant $ K $ such that for all $ Q^\eta \in \Qc_f $
$$\gamma_0 (Q^\eta) \leq K (1 + \Gamma (Q^\eta)).$$
In particular $\inf\limits_{Q^\eta \in \Qc_f} \Gamma(Q^\eta)>-\infty .$
\end{Proposition}
\begin{proof}
Under the assumption (\textbf{A1}) and the nonnegativity of the function $h$, we have
\begin{equation}\label{ineq8}
\begin{split}
\beta E_{Q^\eta} [\dint_0^T \delta_s S_s^\delta (\int_0 ^s h(\eta_u)du) ds + S_T^\delta \int_0^T h(\eta_u)du] &\geq \beta E_{Q^\eta} [S_T ^ \delta \int_0^T h(\eta_u)du]\\& \geq \beta e ^ {- \|\delta \|_{\infty} T} \gamma_0 (Q^\eta).
\end{split}
\end{equation}
Moreover, since $0 \leq S^\delta \leq 1 $, we have:
\begin{equation}\label{ineq9}
\E_{Q^\eta} [\Uc_{0,T}^\delta]\geq -\E_{Q^\eta} [R] =- \E_P[Z_T^\eta R].
\end{equation}
From inequality (\ref{convex inequality2}) where $f(x)=x\log x,$ and as a consequence $f^*(\lambda x)= e^{\lambda x -1}$ we have
\begin{equation} \label{convexinequality}
xy \leq \frac{1}{\lambda} (y \ln y + e^{-1} e^{\lambda x} )\;\; \textrm{for all} \;\; (x,y,\lambda) \in \R \times \R^*_+ \times \R^*
\end{equation}
We get:
\begin{equation*}
\begin{split}
\E_\mathbb{P}[Z_T^\eta R] & \leq \frac{1}{\lambda} \E_P[Z_T^\eta \log Z_T^\eta +e^{-1} e^{\lambda R}]
 = \frac {1}{\lambda} H(Q^\eta| P) + \frac {e^{-1}} {\lambda} \E_P [e^{\lambda R}].
\end{split}
\end{equation*}
From inequality (\ref{estimation entropy}), we deduce that
\begin{equation}\label{ineq10}
\E_P[Z_T^\eta R]\leq \frac{1}{2\lambda \kappa_1} \gamma_0 (Q^\eta) + \frac{T\kappa_2}{2 \lambda \kappa_1} + \frac {e^{-1}} {\lambda} \E_P [e^ {\lambda R}].
\end{equation}
From the definition of $\Gamma(Q^\eta)$, it can be deduced
$$\Gamma(Q^\eta)=E_P[Z_T\Uc_{0,T}]+\beta E_{Q^\eta} [\dint_0^T \delta_s S_s^\delta (\int_0 ^s h(\eta_u)du) ds + S_T^\delta \int_0^T h(\eta_u)du].$$
From (\ref{ineq8}),(\ref{ineq9}) and (\ref{ineq10}), we obtain
\begin{equation*}
\begin{split}
\Gamma (Q^\eta) & \geq \beta e^{-\| \delta \|_{\infty} T} \gamma_0 (Q^\eta) - \frac{1}{2\lambda \kappa_1} \gamma_0 (Q ^\eta) - \frac {T \kappa_2} {2 \lambda \kappa_1} - \frac{e^{-1}}{\lambda} \E_P [e^{\lambda R}]
\\ & = (\beta e^{-\| \delta \|_{\infty} T} - \frac{1}{2 \lambda \kappa_1}) \gamma_0 (Q^\eta)-\dfrac {T \kappa_2} {2 \lambda \kappa_1} - \dfrac {e^{-1}} {\lambda} \E_P [e^{\lambda R}].
\end{split}
\end{equation*}
Choosing $\lambda> 0 $ large enough, there exists $\mu>0$  such that  $ \beta e^{-\| \delta \|_{\infty}T} - \dfrac{1}{2\lambda \kappa_1} \geq \mu. $
From Remark \ref{integrability of U}, we deduce that $\E_P [e ^ {\lambda R}]$ is finite. The desired result is obtained by taking  $ K:=\frac{1}{\mu} \max(1, \dfrac{T \kappa_2}{2 \lambda \kappa_1} + \dfrac {e^{-1}} {\lambda} \E_P [e^{\lambda R }]),$

\ep
\end{proof}
Combining the previous Proposition and the inequality (\ref{estimation entropy}), we obtain the following result.
\begin{Corollary} \label{controlentropy}
Under the assumptions \textbf{(A1)-(A'2)}, there exists  a positive constant $ K'$ such that for all $ Q^\eta \in \Qc_f^c $
$$H(Q^\eta|P) \leq K' (1 + \Gamma (Q^\eta)).$$
\end{Corollary}
To prove the existence of the minimizer probability measure, we need the following technical results which give an upper bound of the utility expectation and show the convexity of our criterion $\Gamma.$
\begin{Lemma} \label{lemma1}
For all $\gamma> 0$ and any $A \in \Fc_T $ we have:
\begin{equation}\label{ineq11}
\E_{Q^\eta} [|\Uc_{0,T}^\delta| \textbf{1}_A] \leq \frac {\gamma_0(Q^\eta)}{2\lambda\kappa_1} + \frac{T\kappa_2}{2\lambda\kappa_1} + \frac{e^{-1}} {\lambda} + \frac {e^{-1}} {\lambda} \E_P [\textbf{1}_A \exp (\lambda \alpha \int_{0}^{T} |U_s| ds+\lambda \bar{\alpha}|\bar{U}_T|)].
\end{equation}
\end{Lemma}
\begin{proof}
From Remark \ref{integrability of U}, we have $R\in L^{\exp}.$ Using the inequality (\ref{convexinequality}), we obtain
\begin{equation*}
\begin{split}
Z_T^\eta |\Uc_{0, T}^\delta| \textbf{1}_A & \leq Z_T^\eta (\alpha \int_{0}^{T} |U_s| ds + \bar{\alpha}|\bar{U}_T|) \textbf{1}_A
\\ & \leq \frac{1}{\lambda} [Z_T^\eta \ln (Z_T^\eta) + e^{-1}\exp (\lambda \alpha \int_{0}^{T} |U_s |ds + \lambda \bar{\alpha}|\bar{U}_T|)] \textbf{1}_A
\\ & \leq \frac{1}{\lambda} [Z_T^\eta \ln (Z_T^\eta) + e^{-1}] + \frac{e^{-1}}{\lambda}\textbf{1}_A \exp(\lambda\alpha \int_{0}^{T}|U_s|ds +\lambda\bar{\alpha}|\bar{U}_T|).
\end{split}
\end{equation*}
Taking the expectation with respect to $P$ and using inequality (\ref{ineqalityentropy}) we get
\begin{equation*}
\E_{Q^{\eta}} [|\Uc_{0, T}^\delta|\textbf{1}_A] \leq \frac{1}{2 \lambda\kappa_1} \gamma_0 (Q^\eta) + \frac{T\kappa_2} {2\lambda \kappa_1} + \frac{e^{-1}} {\lambda} + \frac {e^{-1}} {\lambda} \E_P [\textbf{1}_A \exp (\lambda\alpha\int_{0}^{T}|U_s|ds + \lambda \bar{\alpha}|\bar{U}_T|)],
\end{equation*}
and so inequality (\ref{ineq11}) is proved.
\ep
\end{proof}
\begin{Proposition}\label{convexity of Gamma}
The functional $Q^\eta \mapsto \Gamma (Q^\eta) $ is  convex.
\end{Proposition}
\begin{proof}
The product derivatives formula gives
\begin{equation*}
\frac{d}{ds} (S_s^\delta (\int_0^sh(\eta_u)du)) =-\delta_sS_s^\delta\int_0^sh(\eta_u)du + S_sh(\eta_s).
\end{equation*}
Integrating between $0$ and $T$ we get:
 \begin{equation}\label{ineq12}\dint_0^T \delta_sS_s^\delta (\int_0^sh(\eta_u)du)ds + S_T^\delta \int_0^T h(\eta_u)du = \int_0^T S_s^\delta h(\eta_s)ds. \end{equation}
Fix $\lambda \in(0,1)$ and $ Q^\eta $ and $ Q^{\eta'} $ two distinct elements of $\Qc_f.$
\\ Let $ Q = \lambda Q^\eta + (1-\lambda) Q^{\eta'} $ and $L_t=\E_P[\frac{dQ}{dP}|\Fc_t].$
Using It$\hat{\textrm{o}}$'s formula, we get $ L_t=\Ec(q.W)_t $ where $(q_t)_{0\leq t \leq T}$ is defined by
 $$q_t=\frac{\lambda \eta L_t^\eta+(1-\lambda)\eta'L_t^{\eta'}}{\lambda L_t^\eta+(1-\lambda)L_t^{\eta'}}\textbf{1}_{\{\lambda L_t^\eta+(1-\lambda)L_t^{\eta'}>0\}}\;\; dt\otimes dP\;\textrm{a.e.}\;\;\; t\in [0,T].$$
\\From the definition of the penalty term in $\Gamma$, we have
\begin{equation*}
\begin{split}
\Rc_{0,T}(Q)&=\E_Q\Big[\dint_0^T\delta_sS_s^\delta (\int_0^s h(q_u)du)ds+S_T^\delta\int_0^T h(q_u)du \Big]
\\&=\E_Q\Big[\dint_0^TS_s^\delta h(\frac{\lambda \eta L_s^\eta+(1-\lambda)\eta'L_s^{\eta'}}{\lambda L_s^\eta+(1-\lambda)L_s^{\eta'}}\textbf{1}_{\{\lambda L_s^\eta+(1-\lambda)L_s^{\eta'}>0\}})ds\Big]
\\&= \E_Q\Big[\dint_0^TS_s^\delta h(\frac{\lambda \eta L_s^\eta+(1-\lambda)\eta'L_s^{\eta'}}{\lambda L_s^\eta+(1-\lambda)L_s^{\eta'}})\textbf{1}_{\{\lambda L_s^\eta+(1-\lambda)L_s^{\eta'}>0\}}ds\Big],
\end{split}
\end{equation*}
where the second equality is deduced from $(\ref{ineq12}), $ and the last equality holds because $h(0)=0.$
The convexity of $h$ implies
\begin{equation*}
\begin{split}
&\E_Q\Big[\dint_0^T\delta_sS_s^\delta (\int_0^s h(q_u)du)ds+S_T^\delta\int_0^T h(q_u)du\Big]
\\& \leq \E_Q\Big[\dint_0^TS_s^\delta (\frac{\lambda L^\eta_s}{\lambda L_s^\eta+(1-\lambda)L_s^{\eta'}} h(\eta_s)+\frac{(1-\lambda) L_s^{\eta'}}{\lambda L_s^\eta+(1-\lambda)L_s^{\eta'}} h(\eta'_s))\textbf{1}_{\{\lambda L_s^\eta+(1-\lambda)L_s^{\eta'}>0\}}ds\Big]
\\&=\E_P\Big[\dint_0^T (\lambda L_s^\eta+(1-\lambda)L_s^{\eta'}) S_s^\delta (\frac{\lambda L^\eta_s}{\lambda L_s^\eta+(1-\lambda)L_s^{\eta'}} h(\eta_s)\\ & +\frac{(1-\lambda) L_s^{\eta'}}{\lambda L_s^\eta+(1-\lambda)L_s^{\eta'}} h(\eta'_s))\textbf{1}_{\{\lambda L_s^\eta+(1-\lambda)L_s^{\eta'}>0\}} ds\Big]
\\&=\lambda \E_{Q^\eta}\Big[\dint_0^T S_s^\delta h(\eta_s)\textbf{1}_{\{\lambda L_s^\eta+(1-\lambda)L_s^{\eta'}>0\}}ds]  +(1-\lambda) \E_{Q^{\eta'}}[\dint_0^TS_s^\delta  h(\eta'_s)\textbf{1}_{\{\lambda L_s^\eta+(1-\lambda)L_s^{\eta'}>0\}}ds\Big].
\end{split}
\end{equation*}
Since we have  $\E_Q[\Uc_{0, T}] = \lambda \E_ {Q^\eta} [\Uc_{0, T}]+ (1-\lambda) \E_{Q^{\eta'}} [\Uc_{0, T}], $ we deduce that
 $$ \Gamma (Q) \leq \lambda \Gamma (Q^\eta) + (1-\lambda) \Gamma (Q^{\eta'}).$$
\ep
\end{proof}
The following theorem states the existence of a  probability measure solution of the optimization problem (\ref{min}).
\begin{Theorem} \label{exisun}
Assume that \textbf{(A1)-(A'2)} are satisfied. Then there is a probability measure $Q^{\eta^*}\in\Qc_f^c$ minimizing
$ Q^\eta \mapsto\Gamma(Q^\eta) $  over all $ Q^\eta \in \Qc_f^c.$
\end{Theorem}
\begin{proof}
\begin{enumerate}
\item Let $(Q^{\eta_n})_{n \in \N}$ be a minimizing sequence of $\Qc_f^c$  i.e.
$$ \searrow\lim\limits_{n \to +\infty} \Gamma (Q^{\eta_n}) = \inf \limits_{Q^\eta\in\Qc^c_f} \Gamma(Q^\eta).$$
We denote by $ Z^n:=Z^{Q^{\eta_n}} = \Ec(\int\eta_ndW) $ the  corresponding density processes.
\\Since each $ Z_T^n \geq 0 $, it follows from  Koml$\acute {\textrm{o}}$s' lemma that there is a sequence $ (\bar{Z}_T ^ n) _ {n \in \N} $ such that $ \bar{Z}_T^n \in conv(Z^n_T, Z ^ {n+1} _T, ...)$ for all $ n \in \N $ and $ (\bar{Z}^n_t) $ converges $P$-a.s to a random variable $\bar{Z}^{\infty}_T $.
\\ $ \bar{Z}^{\infty}_T $ is positive but may be infinite.
As $ \Qc_f$ is convex, each $\bar {Z}_T^n$ is associated with a probability measure $ \bar{Q}^n \in \Qc_f.$
This also holds for $ \bar {Z}^{\infty}_T $ i.e.  that $ d\bar{Q}^{\infty}:=\bar{Z}^{\infty}dP $ defines a probability measure $ \bar{Q}^{\infty} \in \Qc_f.$
Indeed, we have first:
\begin{equation}\label{ineq13}
\Gamma (\bar{Q}^n) \leq \sup \limits_{m\geq n} \Gamma (Q^{\eta_m}) = \Gamma(Q^{\eta_n}) \leq \Gamma (Q^{\eta_1}),
\end{equation}
where the first inequality holds since $Q^\eta\mapsto \Gamma(Q^\eta)$ is convex and $ n \mapsto \Gamma(Q^{\eta_n}) $ is decreasing,  and the second inequality follows from the monotonicity property of $(\Gamma(Q^{\eta_n}))_n$. Therefore Corollary \ref{controlentropy} gives
\begin{equation} \label{c2} \sup\limits_{n \in \N} \E_P [\bar{Z}^n\ln(\bar{Z}^n)] = \sup\limits_{n\in\N} H(Q^n|P) \leq K'(1+\sup \limits_{n\in\N}\Gamma (\bar{Q}^n)) \leq K'(1+ \Gamma (Q^{\eta_1})).
\end{equation}
Thus $ (\bar{Z}^n_T)_{n\in\N} $ is $ P $-uniformly integrable by Vall\'ee-Poussin's criterion and converges in $ L^1(P)$. This implies that $\E_P[\bar{Z}^{\infty}_T]=\lim \limits_{n \rightarrow +\infty} E_P[\bar{Z}^n_T] = 1 $ so that $ Q^{\infty} $ be a probability measure and $ Q^{\infty} \ll P$ on  $\Fc_T.$ Let define the martingale $Z_t^\infty:=E_P[Z_T^\infty|\Fc_t]$, so there exists a progressively measurable process $(\eta_t^\infty)_t$  valued in $\R^d$ satisfying $\int_0^T\Vert\eta_t^\infty\Vert^2 dt <+\infty\;\; P.a.s$  and $Z_t^\infty=\Ec(\int_0^t \eta_s^\infty dW_s)$ . Similarly, for $n\in \mathbb{N}$, there exists a progressively measurable process $(\bar\eta^n_t)_t$  valued in $\R^d$ satisfying $\int_0^T\Vert\bar\eta_t^n\Vert^2 dt <+\infty\;\; P.a.s$ and $\bar Z^n_t=\Ec(\int_0^t \bar\eta_s^n dW_s).$
\item We now want to show that $\bar{Q}^\infty \in \Qc_f^c.$ Let $ \bar{Z}^\infty $ be the density process of  $ \bar{Q}^\infty$ with respect to $P$.
Since we know that $ (\bar{Z}^n_T) $ converges to $ \bar{Z}^\infty $ in $ L^1(P),$ the maximal Doob's inequality
$$P[\sup\limits_{0\leq t \leq T} \mid \bar{Z}^\infty_t-\bar{Z}^n_t \mid\geq \epsilon]\leq \frac{1}{\epsilon} \E_P[\mid \bar{Z}^\infty_T-\bar{Z}^n_T \mid]$$
implies that  $(\sup\limits_{0\leq t \leq T} \mid \bar{Z}^\infty_t-\bar{Z}^n_t \mid)_{n\in\N}$
converges to $0$ in $P$-probability.
Going to a subsequence, still denoted by $(\bar{Z}^n)_{n\in \N}$, we can assume that $(\sup\limits_{0\leq t \leq T} \mid \bar{Z}^\infty_t-\bar{Z}^n_t \mid)_{n\in\N}$ converges to $0$  $P$-a.s. By Burkholder-Davis-Gundy's inequality there is a constant $C$ such that
$$E[\langle \bar{Z}^\infty-\bar{Z}^n \rangle_T^\frac{1}{2}]\leq C E[\sup\limits_{0\leq t \leq T} \mid \bar{Z}^\infty_t-\bar{Z}^n_t \mid].$$
\\Let $M_t^n:=\sup\limits_{0\leq s \leq t} \mid \bar{Z}^\infty_s-\bar{Z}^n_s \mid $ and $(\tau_n)$ a sequence of stopping time defined by
\begin{equation*}
\tau_n=\left\{\begin{array}{ccc}
         \inf\{t\in [0,T[; M^n_t\geq 1\} & \textrm{if}  & \{t\in [0,T[; M^n_t \geq 1\}\neq \emptyset \\
         T &  & \textrm{otherwise}
       \end{array}\right..
\end{equation*}
Since $M_{\tau_n}^n$ is bounded by $M_T^n \wedge 1$ then $M_{\tau_n}^n$ converges almost surely to $0$ and by the dominated convergence theorem converges to $0$ in $L^1(P).$
Then, using Burkholder-Davis-Gundy's inequality $\langle \bar{Z}^\infty-\bar{Z}^n \rangle_{\tau_n}^\frac{1}{2}$ converges to $0$ in $L^1(P)$ and a fortiori in probability.
\\ As  $\langle \bar{Z}^\infty-\bar{Z}^n \rangle_T= \langle \bar{Z}^\infty-\bar{Z}^n \rangle_{\tau_n}\textbf{1}_{\{\tau_n=T\}}+\langle \bar{Z}^\infty-\bar{Z}^n \rangle_T\textbf{1}_{\{\tau_n<T\}}$, then for all $\varepsilon>0,$
\begin{equation*}
\begin{split}
P(\langle \bar{Z}^\infty-\bar{Z}^n \rangle_T\geq \varepsilon) & \leq P(\langle \bar{Z}^\infty-\bar{Z}^n \rangle_{\tau_n}\textbf{1}_{\{\tau_n=T\}}\geq \varepsilon)+ P(\langle \bar{Z}^\infty-\bar{Z}^n \rangle_T\textbf{1}_{\{\tau_n<T\}}\geq \varepsilon)
\\ & \leq P(\langle \bar{Z}^\infty-\bar{Z}^n \rangle_{\tau_n}\geq \varepsilon)+P(\tau_n<T)
\end{split}
\end{equation*}
From the convergence in probability of $(\langle \bar{Z}^\infty-\bar{Z}^n \rangle_{\tau_n})_n$, we have $\lim\limits_{n\rightarrow +\infty}P(\langle \bar{Z}^\infty-\bar{Z}^n \rangle_{\tau_n}\geq \varepsilon)=0.$
Since $M^n$  is a nondecreasing process, we have
$$P(\tau_n<T)=P(\{\exists t\in[0,T[\;\; s.t\;\; M^n_t\geq 1\}) \leq P(\{ M^n_T\geq 1\}).$$
Since $M_T^n$ converges in probability to $0$, we have $P(\{ M^n_T\geq 1\}) \underset{n\rightarrow +\infty}{\longrightarrow} 0$. Then $\lim\limits_{n\rightarrow +\infty}P(\tau_n<T)=0,$
and consequently $\lim\limits_{n\rightarrow +\infty}P(\langle \bar{Z}^\infty-\bar{Z}^n \rangle_T\geq \varepsilon)=0$ .ie $(\langle \bar{Z}^\infty-\bar{Z}^n \rangle_T)_n $ converges in probability to $0.$
We can extract a subsequence also  denoted  by $\bar{Z}^n$ such that  $(\langle \bar{Z}^\infty-\bar{Z}^n \rangle_T)_n $ converges  almost surely to $0.$
\\On the other hand, we have $$\langle \bar{Z}^\infty-\bar{Z}^n \rangle_T=\dint_0^T(\bar{Z}_u^\infty\bar{\eta}_u^\infty-\bar{Z}^n_u\bar{\eta}^n_u)^2du .$$
It follows that processes $\bar{Z}^n\bar{\eta}^n$ converge in $dt\otimes dP$-measure to process $\bar{Z}^\infty\bar{\eta}^\infty.$ Since $\bar{Z}^n\longrightarrow \bar{Z}^\infty dt\otimes dP$-a.e, we have $\bar{\eta}^n$ converges in $dt\otimes dP$-measure to $\bar{\eta}^\infty.$
 Fatou's lemma and inequality (\ref {c2}) give:
\begin{equation}\label{c3}
 \gamma_0(\bar{Q}^\infty) = \mathbb{E}_P [ \bar{Z}^\infty_T\int_0^T h(\bar{\eta}_u^\infty)du] \leq \liminf\limits_{n\rightarrow +\infty} \E_\mathbb{P} [Z^n_T\int_0^T h(\eta_u^n) du] < +\infty.
\end{equation}
This shows that $\bar {Q}^{\infty} \in \Qc_f.$
\\Now we will show that the probability $\bar {Q}^{\infty}$ is optimal.
\\For $n\in\N \cup \{+\infty\}$, let $\bar{Y}_1^n:=\bar {Z}^n_T \Uc^\delta_{0,T}$ and $ \bar{Y}_2^n:=\beta \Rc^\delta_{0,T}(\bar{Q}^n))$
 then
$\lim\limits_{ n\rightarrow +\infty} \bar{Y}^n_i = \bar{Y}^\infty_i \;P-$a.s for $i = 1,2.$
As $\bar{Y}^n_2 $ is bounded from below, uniformly in $n$ and $\omega$,  Fatou's lemma yields:
\begin{equation}\label{c4}
\E_\mathbb{P}[\bar{Y}^\infty_2]\leq \liminf\limits_{n\rightarrow \infty}\E_\mathbb{P}[\bar{Y}^n_2].
\end{equation}
Adopting the same approach as in Theorem \ref{exisun} we show that:
\begin{equation}\label{c5}
\E_\mathbb{P}[\bar{Y}^\infty_1]\leq \liminf\limits_{n\rightarrow \infty}\E_\mathbb{P}[\bar{Y}^n_1].
\end{equation}
Inequality(\ref{c4}), (\ref{c5}) and (\ref {c2}) provide that:
$$\Gamma(\bar{Q}^\infty) =\E_\mathbb{P} [ \bar{Y}^\infty_1+\bar{Y}^\infty_2 ] \leq \liminf\limits_{n\rightarrow \infty}\Gamma(\bar{Q}^n)
\leq \liminf\limits_{n\rightarrow \infty}\Gamma(Q^n)\leq \inf\limits_{Q\in \Qc_f}\Gamma(Q).$$
This proves that $\bar{Q}^\infty $ is indeed optimal.
\end{enumerate}
\ep
\end{proof}
\subsection{BSDE description for the dynamic value process}
In this section,  stochastic control techniques  are employed to study the dynamics of the  value process denoted by $V$ associated with the optimization problem (\ref{min}). It is proved that $V$ is the unique solution of a quadratic backward stochastic differential equation. This extends the work of Skiadas \cite{SK03}, Schroder and Skiadas \cite{SS03}.
\\We first introduce some notations that we use below.
Denote by $\Sc$ the set of all $\Fc$ stopping time $\tau $ with values in $ [0,T] $, $\Dc^c $ the space of all processes $\eta $ with $ Q^\eta \in \Qc_f^c$ and $\Dc^{c,e} $ the space of all processes $\eta $ with $ Q^\eta \in \Qc_f^{c,e}$. We define:
$$ \mathcal{D}^c(\eta, \tau): = \{\eta '\in \Dc^c, Q^\eta = Q^{\eta'} \, \textrm {on} \, [0,\tau] \} $$
$$ \Gamma (\tau,Q^\eta):= \mathbb{E}_{Q^{\eta}} [c (.,Q^\eta)| \Fc_\tau]. $$
We note that $\Gamma (0,Q^\eta)$ and $\Gamma (Q^\eta)$ coincide.
The minimal conditional cost at time $\tau$ is defined by
$$ J(\tau,Q^\eta): = Q^\eta-\underset{\eta'\in \Dc^c(\eta,\tau)}{\essinf} \Gamma (\tau,Q^{\eta'}). $$
Then the problem (\ref{min}) can be written as follows:
\begin{equation}
 \textrm{give} \inf\limits_{Q^\eta\in\Qc_f^c}\Gamma(Q^\eta)=\inf \limits_{Q^\eta\in\Qc_f^c}\E_{Q^\eta}[c(.,Q^\eta)]= \E_P[J(0,Q^\eta)].
 \end{equation}
Where the second equality is deduced since the  dynamic programming principle holds and we have $Q^\eta=P$ on $\Fc_0$ for all $Q^\eta \in\Qc_f^c.$
\\The following  martingale optimality principle is a direct consequence of Theorems 1.15, 1.17 and 1.21 in El Karoui\cite{ELK81}. For the sake of completeness, the proof is given in the Appendix.
\begin{Proposition} \label{martingaleoptimality}
\begin{itemize}
\item[(1)]The family $ \{J (\tau,Q^\eta) | \tau \in \mathcal{S}, Q^\eta \in Q_f^c\} $ is a submartingale system.
\item[(2)] $ Q^{\eta^*} \in \Qc_f^c$  is optimal $\Leftrightarrow $ $ \{J (\tau, Q^{\eta^*}) | \tau \in \Sc \}$ is a martingale system .
\item[(3)] For any $Q^{\eta} \in \Qc_f^c$  there is an RCLL adapted process
$(J^\eta_t)_{0 \leq t \leq T}$ which is a $Q^\eta$- martingale  $J_{\tau}^\eta = J(\tau,Q^\eta)$.
\end{itemize}
\end{Proposition}
In order to characterize the value process  in terms of BSDE we need the following proposition.
\begin{Proposition}\label{infQeq}
Under \textbf{(A1)-(A'2)}, we have
\begin{equation*}
\inf\limits_{Q^\eta\in \Qc_f^c}\Gamma(Q^\eta)=\inf\limits_{Q^\eta\in \Qc_f^{c,e}}\Gamma(Q^\eta).
\end{equation*}
\end{Proposition}
\begin{proof}
 Let $Q^{\eta^*} \in \Qc_f^c$ such that $\inf\limits_{Q^\eta\in \Qc_f^c}\E_{Q^\eta}[c(.,Q^\eta)]=\E_{Q^{\eta^*}}[c(.,Q^{\eta^*})]$ and $\lambda \in [0,1)$,  then $\lambda Q^{\eta^*} +(1-\lambda) P \in \Qc_f^{c,e}.$
Since $Q^\eta\mapsto \Gamma(Q^\eta)$ is convex then
$$\Gamma(\lambda Q^{\eta^*} +(1-\lambda) P)\leq \lambda\Gamma( Q^{\eta^*}) +(1-\lambda) \Gamma(P) \;\; \forall \lambda \in [0,1),$$
which implies $$\limsup\limits_{\lambda\rightarrow 1}\Gamma(\lambda Q^{\eta^*} +(1-\lambda) P)\leq \Gamma( Q^{\eta^*}). $$
Consequently, we have
$$\inf\limits_{Q^\eta\in \Qc_f^c}\Gamma(Q^\eta)\geq\inf\limits_{Q^\eta\in \Qc_f^{c,e}}\Gamma(Q^\eta).$$
The converse  inequality holds since $\Qc_f^{c,e} \subset \Qc_f^c.$\\
\ep
\end{proof}

 We later use  a strong order relation on the set of increasing processes defined by
\begin{Definition}
Let $A$ and $B$ two increasing process. We say $A\preceq B$ if the process $B-A$ is increasing.
\end{Definition}
We already know from Theorem \ref{exisun} that there is an optimal model  $Q^{\eta^*}\in\Qc_f^c$.
For each $ Q^\eta \in \Qc^{c,e}_f $ and $ \tau \in \Sc $, we  define the value of the control problem starting at time $\tau$
$$ V (\tau, Q^\eta)= Q^\eta-ess\inf\limits_ {{\eta'}\in \Dc^c(\eta,\tau)}\tilde{V} (\tau, Q^{\eta'}), $$
where
$$ \tilde{V}(\tau,Q^{\eta'}) = \E_{Q^{\eta'}} [\Uc^\delta_ {\tau,T} | \Fc_ \tau] + \beta \E_{Q^{\eta'}} [\Rc_{\tau, T}^\delta(Q^{\eta'}) | \Fc_ \tau]. $$
We need to define the following space
$$\Hc^p_d=\Big\{ (Z_t)_{0\leq t \leq T}\mathbb{F}\textrm{-progressively measurable process valued in}\;\; \R^d\;\; s.t \;\;\E_P[(\int_0^T|Z_u|^2du)^\frac{p}{2}]<\infty\Big\}.$$
Before stating the main result of this section, we recall a result on the existence and uniqueness of a family of BSDE due to Briand
and Hu \cite{BH07} (see also Barieu and El Karoui \cite{Bar13}).
\begin{Theorem}
We assume that there exist two constants $\mu>0$  and $\nu > 0  $ together with a nonnegative
progressively measurable stochastic process ${(\rho_t)}_{0\leq t\leq T}$ such that, $P-$a.s.,
\begin{enumerate}
\item[(i)]for all $t \in [0, T ]$, for all $y \in \mathbb{R}$, $z \longmapsto f (t, y, z)$ is convex;
\item[(ii)] for all $(t,z) \in [0,T ] \times \mathbb{R} ,(y, y') \in \mathbb{R}^2 , |f (t,y,z) - f (t,y',z)|\leq \nu |y - y'| $
\item[(iii)] $f $ has the following growth: $$|f (t, y, z)| \leq \rho_t + \nu|y| +\mu |z|^2 ; \forall (t, y, z) \in [0, T ]
\times \mathbb{R} \times \mathbb{R}^d,$$
\item [(iv)]$|\rho|_1:=\int_0^T|\rho_t|dt$ and $|\xi|$ have exponential moments of all order.
\end{enumerate}
 Then the BSDE
 $$
Y_t=\xi + \int_t^T f(s,Y_s,Z_s)\,ds-\int_t^T Z_s dW_s,\qquad 0\leq t\leq T,
$$
 has a unique
solution $(Y,Z)$ such that $Y$ belongs to $L^{exp}$ and $Z$ belongs to $\Hc^p_d$ for each $p\geq 1.$

\end{Theorem}
The following result characterizes  value process $V$ as the unique solution of a BSDE with a quadratic generator and unbounded terminal condition.
Precisely we have
\begin{Theorem}
Under the assumptions \textbf{(A1)-(A'2)},  pair $(V,Z)$ is the unique solution in $D_0^{\exp}\times \Hc^p_{d}, p\geq 1, $ of the following BSDE:
\begin{equation}\label{bsde description}
\left \{\begin{split}
&dY_t  = (\delta_tY_t-\alpha U_t+\beta h^*(\frac{1}{\beta}Z_t))dt - Z_tdW_t,
\\ & Y_T = \bar{\alpha} \bar{U}_T.
\end{split} \right.
\end{equation}
and $Q^*$ is equivalent to $P.$
\end{Theorem}
\begin{proof}
\\\underline{First step: the process $V$ satisfies the BSDE(\ref{bsde description})}
\\ By using Bayes' formula and the definition of $ \Rc_ {\tau,T}^\delta(Q^{\eta'})$, it is clear that $\tilde {V}(\tau,Q^{\eta'})$ depends only on the values of $ \eta'$ on $(\tau, T]$ and is therefore independent of  $Q^\eta$ since $Q^\eta=Q^{\eta'}$ on $\Fc_\tau.$
Thus we can also take the ess$\inf$ under $P\approx Q^\eta $ .
From Proposition \ref{infQeq}, we could take the infimum over the set $\Qc_f^{c,e},$  which implies
 $$ V (\tau, Q^\eta)= P-\underset{{\eta'}\in \Dc^c(\eta,\tau)\cap \Dc^{c,e}}{\essinf}\tilde{V} (\tau, Q^{\eta'}),$$ for all $Q^\eta\in \Qc_f^{c,e}.$
 Since $V(\tau,Q^\eta )$ is independent of $Q^\eta$, we can denote $V(\tau,Q^\eta )$ by $V(\tau).$
\\We fix $\eta'\in \Dc(Q^\eta,\tau)$. From the definition of $\Rc_{t,T}^\delta (Q^{\eta'}) $ (see equation (\ref{penalty definition ctc})), we have
\begin{equation*}
\begin{split}
\Rc_{0, T}^\delta (Q ^{\eta'})& = \dint_0^T \delta_sS_s ^ \delta (\int_0^sh (\eta'_u)du)ds + S_T^\delta \int_0^T h(\eta'_u)du
\\& = \dint_0^\tau \delta_sS_s^\delta(\int_0^sh(\eta_u)du)ds+S_\tau^\delta \int_0^\tau h(\eta_u)du + S_\tau^\delta \Rc_{\tau,T}^\delta(Q^{\eta'}).
\end{split}
\end{equation*}
By comparing the definitions of $ V(\tau)=V(\tau, Q^\eta) $ and $J_\tau^\eta,$ then we get for $ Q^\eta \in \Qc^{c,e}_f $
\begin{equation}\label{jtau}
J^\eta_\tau = S^\delta_\tau V_\tau + \alpha \int_0^\tau S^\delta_sU_sds + \beta \Big(\dint_0^\tau \delta_sS_s^\delta (\int_0 ^sh (\eta_u)du)ds + S_\tau^\delta \int_0^\tau h(\eta_u)du\Big).
\end{equation}
Arguing as above, the ess$\inf$ for $ J_\tau^\eta $ could be taken under $ P \approx Q^\eta $.
From the Proposition \ref{martingaleoptimality}, $J^\eta_\tau$ admits an RCLL version. From equality (\ref{jtau}), an appropriate RCLL process $V=(V_t)_{0 \leq t \leq T }$ can be chosen  such that
$$ V_\tau = V(\tau) = V(\tau,Q^\eta), \; \mathbb{P}.a.s \,\;\; \textrm{for all}\;\; \tau \in \Sc \, \, \textrm {and} \, \, Q^\eta \in \Qc_f^{c,e} $$
and then we have for all $ Q^\eta \in \Qc^{c,e}_f $
\begin{equation}\label{jq}
J_t^\eta=S_t^\delta V_t + \alpha \int_0^t S^\delta_sU_sds + \beta \Big(\dint_0^t \delta_sS_s^\delta (\int_0^sh (\eta_u) du) ds + S_t^\delta \int_0^t h(\eta_s) ds \Big)\;\; dt\otimes dP\; \textrm{a.e},\; 0\leq t\leq T.
\end{equation}
 If we take $\eta\equiv 0, $ the probability measure $Q^0$ coincides with the historical probability measure $P.$ Then, by the Proposition \ref{martingaleoptimality}, $J^0 $ is $ P$- submartingale. From equation (\ref{jtau}),
$ J^0=S^\delta V+\alpha \int S^\delta_sU_sds $ and thus, by It$\hat{\textrm{o}}$'s lemma,  it can be deduced  that $V$ is a $P$-special semimartingale. Its canonical decomposition can be written  as follows:
\begin{equation}\label{canonicdecomp}
 V_t= V_0 - \int_0^t q_sdW_s + \int_0^t K_sds .
 \end{equation}
For each $ Q^\eta \in \Qc^{c,e}_f$, we have $Z_.^\eta=\Ec(\int_0^. \eta dW). $ Plugging (\ref{canonicdecomp}) into (\ref{jtau}), we obtain
\begin{equation*}
dJ^\eta_t =  S^\delta_t (-q_tdW_t+K_tdt) - \delta_tS_tV_tdt + \alpha S_tU_tdt + \beta  S_t^\delta h(\eta_t) dt.
\end{equation*}
\\ By Girsanov's  theorem the process $-\int_0 ^. q_tdW_t+\int_0 ^.q_t \eta_tdt $ is a local martingale under $Q^\eta$ and the dynamic of $(J^\eta_t)_t$ is given by:
\begin{equation*}
dJ^\eta_t = S^\delta_t (-q_tdW_t+q_t \eta_tdt) + S^\delta_t (K_t- q_t\eta_t + \beta h (\eta_t)) dt-\delta_tS^\delta_tV_tdt + \alpha S^\delta_tU_tdt.
\end{equation*}
$J^\eta $ is a $Q^\eta-$ submartingale  and $J^{\eta^*}$ is $ Q^{\eta^*}$- martingale. Such properties hold if we choose $K_t= \delta V_t-\alpha U_t - \mbox{ess}\inf\limits_{\eta}(-q_t\eta_t+\beta h(\eta_t)),$
 where the essential infimum  is taken in the sense of strong order $\preceq .$
 Therefore
 \begin{equation}\label{processK}
 K_t= \delta V_t-\alpha U_t + \mbox{ess}\sup\limits_{\eta}(q_t\eta_t-\beta h(\eta_t))=
  \delta_t V_t-\alpha U_t + \beta h^*(\frac{1}{\beta}q_t).
 \end{equation}
\\ This ess $\inf$ is reached for  $\eta^*_t$ in the subdifferential of $h^*$ at $\displaystyle\frac{1}{\beta}q_t.$
From (\ref{canonicdecomp}) and (\ref{processK}) we deduce that:
\begin{equation*}\left\{ \begin{split}
   & dV_t=(\delta_tV_t-\alpha U_t +\beta h^*(\frac{1}{\beta}q_t))dt - q_tdW_t  \\
 &  V_T=\bar{\alpha}\bar{U}_T.
 \end{split}
 \right.
\end{equation*}
\underline{Second step: The minimal probability measure $Q^{\eta^*}$ is equivalent to $P$}
\\Since
$q_t\eta^*_t-\beta h(\eta^*_t)=\displaystyle\beta h^*(\frac{1}{\beta}q_t)$ , we have
$$ h(\eta^*_t)=\frac{1}{\beta}[q_t\eta^*_t-\beta h^*(\frac{1}{\beta}q_t)].$$
Thus, we have
\begin{equation*}
\begin{split}
\kappa_1\|\eta^*_t\|^2-\kappa_2\leq |h(\eta_t^*)|& \leq \frac{1}{\beta}\|q_t\eta_t^*\|+|h^*(\frac{1}{\beta}q_t)|
\\&\leq \frac{1}{\beta}(\epsilon^2\|q_t\|^2+\frac{1}{\epsilon^2}\|\eta^*_t\|^2)+|h^*(\frac{1}{\beta}q_t)|
\\&\leq \frac{1}{\beta}(\epsilon^2\|q_t\|^2+\frac{1}{\epsilon^2}\|\eta^*_t\|^2)+\frac{1}{4\kappa_1}\|\frac{1}{\beta}q_t\|^2+\kappa_2.
\end{split}
\end{equation*}
The last inequality is a consequence of the fact that
$$
h(x)\geq \kappa_1\|x\|^2-\kappa_2,
$$
implies
$$h^*(x)\leq \frac{1}{4\kappa_1}\|x\|^2+\kappa_2.$$
Therefore,
\begin{equation*}
\begin{split}
(\kappa_1-\frac{1}{\beta\epsilon^2})\|\eta^*_t\|^2\leq (\frac{\epsilon^2}{\beta}+\frac{1}{4\kappa_1\beta^2})\|q_t\|^2+2\kappa_2.
\end{split}
\end{equation*}
By choosing $\epsilon$ large enough such that $\displaystyle \kappa_1-\frac{1}{\beta\epsilon^2}$ is strictly positive, there exists $C_1>0,$ $C_2 \in \mathbb{R}$ such that $$\|\eta^*_t\|^2\leq C_1 \|q_t\|^2+C_2.$$
The process $V$ is a $P$-special semimartingale, then $\int_0^T\|q_t\|^2dt <\infty P.a.s$
which implies that $P\big(\{\displaystyle\frac{dQ^{\eta^*}}{dP}|_{\Fc_T}=0\}\big)=P\big(\{\int_0^T\|\eta^*_t\|^2dt=\infty\}\big)=0$. Hence the probability measure $Q^{\eta^*}$ is equivalent to $ P.$
\\\underline{Third step: the process $V$ lies in $D_0^{\exp}.$}
Because $D_0^{\exp}$ is a vector space, it is enough to prove that $V^+$ and $V^-$ lie both in it.
\\We first show that the process $V^+$ is in $D_0^{\exp}.$
By definition of the utility term, we have $$V_t \leq \E_ {P} [\Uc^\delta_{t,T} | \Fc_\tau]  \leq \E_ {P} [\Uc^\delta_{0,T} | \Fc_\tau].$$
\\Fix $\gamma >0$ and choose an RCLL version of the $P$- martingale $N$ defined by  $N_t:=E_P[e^{\gamma |\Uc^\delta_{0,T}|}|\mathcal{F}_t].$
\\Then the continuity of $V$ (see \ref{canonicdecomp} ) and Jensen's inequality imply that
\begin{equation}\label{ineq14}
\begin{split}
\exp(\gamma \esssup_{0\leq t \leq T}V^+_t)&=\exp(\gamma \sup\limits_{0\leq t \leq T}V_t^+)
\\& \leq \sup\limits_{0\leq t \leq T}N_t.
\end{split}
\end{equation}
Since $S^\delta_T \leq 1,$ we have
$$|\Uc_{0,T}^\delta|\leq \alpha\int_0^T|U_s|ds+\bar{\alpha}|\bar{U}_T|=R.$$
Since $e^{\gamma R}\in L^p(P)$ for every $p\in (0,+\infty)$, Doob's inequality imply that $\sup\limits_{0\leq t \leq T}N_t$ is in $L^p(P)$ for every $p\in (1,+\infty)$. Hence the result follows from (\ref{ineq14}).
\\It remains to show that  $V^-$ is also in $D_0^{\exp}.$
For this reason we use the integrability results obtained by  Bordigoni, Matoussi and Schweizer \cite{BMS05}
 in the context of an entropic Penalty  i.e. when the penalty term is given by
\begin{equation}\label{penalty entropic definition ctc} \Rc_{t, T}^{\delta,H}(Q^\eta)= \dint_t^ T \delta_s \frac{S_s^\delta}{S_t ^ \delta} (\frac{1}{2}\int_t^s \|\eta_u\|^2) du) ds + \frac{S_T^\delta} {S_t^\delta}\frac{1}{2} \int_t^T \|\eta_u\|^2 du,  \; \forall  \, 0 \leq t \leq T, \end{equation}
Since $h(x)\geq\kappa_1 |x|^2-\kappa_2, $ we have $\displaystyle \frac{1}{2}|\eta_u|^2\leq \frac{1}{2\kappa_1}h(\eta_u)+\frac{\kappa_2}{2\kappa_1},$ which implies
\begin{equation*}
\begin{split}
\Rc_{t, T}^{\delta,H}(Q^\eta)&\leq \dint_t^ T \delta_s \frac{S_s^\delta}{S_t ^ \delta} (\frac{\kappa_2}{2\kappa_1}\int_t^s du) ds + \frac{S_T^\delta} {S_t^\delta}\frac{\kappa_2}{2\kappa_1} \int_t^T  du
\\&+ \dint_t^ T \delta_s \frac{S_s^\delta}{S_t ^ \delta} (\frac{1}{2\kappa_1}\int_t^s h(\eta_u) du) ds + \frac{S_T^\delta} {S_t^\delta}\frac{1}{2\kappa_1} \int_t^T h(\eta_u) du.
\\& \leq \frac{\kappa_2}{2\kappa_1}T(1+T\|\delta\|_\infty)+\frac{1}{2\kappa_1}\Rc_{t,T}^{\delta}(Q^\eta).
\end{split}
\end{equation*}
Using the definition of $V$, we obtain
$$V_t\geq V^H_t-\beta\kappa_2T(1+T\|\delta\|_\infty)\;\; dt\otimes dP.a.s,$$
where $V^H$  is the value process when the penalty is entropic and the parameter $\beta$  is replaced by $\displaystyle 2\beta\kappa_1.$
\\And Consequently $$V^-_t\leq (V^H_t)^-+\beta\kappa_2T(1+T\|\delta\|_\infty) dt\otimes dP.a.s.$$
\\By Bordigoni, Matoussi and Schweizer \cite{BMS05}, we have $(V^H)^- \in D_0^{\exp}, $ this implies that $V^- \in D_0^{\exp}.$
\ep
\end{proof}
\begin{Remark}
According to  Briand and Hu \cite{BH07} the equation (\ref{bsde description}) has a unique solution because
$$h^*(x)\leq \frac{1}{4\kappa_1}\|x\|^2+\kappa_2$$
and hence the driver $f$ of BSDE (\ref{bsde description}) given by $f(t,w,y,z)=\delta_ty-\alpha U_t +\beta h^*(\frac{1}{\beta}z)$ satisfies:
\begin{enumerate}
\item for all $t\in [0,T]$, for all $y\in \R,$ $z\mapsto f(t,y,z)$ is convex;
\item for all $(t,z)\in[0,T]\times \R^d$,
$$\forall (y,y')\in \R^2; |f(t,y,z)-f(t,y',z)|\leq \parallel \delta\parallel_\infty|y-y'|.$$
\item for all $(t,y,z)\in[0,T]\times\R \times \R^d$, \begin{equation*}
\begin{split}
|f(t,y,z)|&=|\delta_ty-\alpha U_t +\beta h^*(\frac{1}{\beta}z)|
\leq\parallel\delta\parallel_\infty |y|+|\alpha| |U_t|+\frac{1}{4\kappa_1\beta}\|z\|^2+\beta\kappa_2.
\end{split}
\end{equation*}
\end{enumerate}
Since the process $ U\in D_1^{\exp}$ and the terminal condition $\bar \alpha\bar{U}_T$ belongs to $L^{\exp}$, the existence of the BSDE solution is insured. The uniqueness result is a direct consequence of  the convexity proprety of $h^*.$
\end{Remark}
\subsection{A comparison with related results}
In the case of the entropic penalty which corresponds to $h(x)=\frac{1}{2}\|x\|^2,$ the value process  is described through the backward stochastic differential equation:
\begin{equation}\label{bsde description entropic case}
\left \{\begin{split}
dY_t & = (\delta_tY_t-\alpha U_t+ \frac{1}{2\beta}\|Z\|^2_t)dt - Z_tdW_t,
\\  Y_T& = \bar{\alpha}\bar{U}_T.
\end{split} \right.
\end{equation}
These results are obtained by Schroder and Skiadas in \cite{SK03,SS03} where $\bar{\alpha}=0.$
In the context of a dynamic concave utility, Delbaen, Hu and Bao \cite{Delb09} treated  the case $\alpha=0,\delta=0,\beta=1$ and$\xi=\bar{\alpha}\bar{U}_T$ is bounded. In this special case the existence of an optimal probability is a direct consequence of Dunford-Pettis' theorem and James' theorem  shown in Jouini-Schachermayer-Touzi's work \cite{JST05}.
Delbaen et al. showed that the dynamic concave utility $$Y_t= ess\inf\limits_{Q\in \Qc_f}E_Q[\xi+\int_t^Th(\eta_u)du|\mathcal{F}_t]$$
satisfies the following BSDE:
\begin{equation}
\left \{\begin{split}
&dY_t  = h^*(Z_t)dt - Z_tdW_t,
\\ & Y_T = \xi.
\end{split} \right.
\end{equation}
$\hfill \Box $

\subsection{Example : Portfolio  and consumption choice}\label{exampleprotfolio}
First, we note that the characterization of the value function using BSDE allows us to establish a connection between robust  utlity and recursive utility. This class of non-time-separable utility  has been studied by a number of authors. Kreps and Porteus \cite{KP78},  Epstein and  Zin\cite{EZ89}  analyzed this type of utility in a discrete-time setting,
while Duffie and  Epstein \cite{DE92} studied the continuous-time case.
In the following, we propose an exemple  with portfolio choice where  we take $\delta \equiv 0$  and $h(x)=\kappa \|x\|^2$,  $\kappa \in \mathbb{R}_+^*.$
Then, we have $h^*(x)=\frac{1}{4\kappa} \|x\|^2$ and so the value function of our control problem satisfies the dynamics:
\begin{equation}\label{bsde desc 1}
\left \{\begin{split}
dY_t  &= (-\alpha U_t+\frac{1}{4\kappa\beta} \|Z_t\|^2)dt - Z_tdW_t,
\\  Y_T &= \bar{\alpha} \bar{U}_T.
\end{split} \right.
\end{equation}
We consider the process  $(\Gamma_t)_{t\in [0,T]}$ defined by $\displaystyle \Gamma_t=Y_t+\int_0^t \alpha U_s ds.$ This process satisfies the following BSDE
\begin{equation}\label{bsde desc 2}
\left \{\begin{split}
d\Gamma_t & = \frac{1}{4\kappa\beta} \|Z_t\|^2dt -  Z_tdW_t,
\\  \Gamma_T & = \bar{\alpha} \bar{U}_T+\int_0^T \alpha_s U_s ds.
\end{split} \right.
\end{equation}
The solution of the BSDE (\ref{bsde desc 2}) is given by $\Gamma_t=\displaystyle -2\kappa\beta \log E_P[\exp(-\frac{1}{2k\beta}(\bar{\alpha}\bar{U}_T+\int_0^T\alpha U_s ds))|\mathcal{F}_t],$ and consequently $$Y_t=\displaystyle -\bar{\beta} \log E_P[\exp(-\frac{1}{\bar{\beta}}(\bar{\alpha}\bar{U}_T+\int_0^T\alpha U_s ds))|\mathcal{F}_t]-\int_0^t \alpha U_s ds,$$ where $\bar{\beta}=2\kappa\beta .$
Then, the problem of robust  utility is simply a problem of recursive utility.

We are now interested  in the  utility maximization problem in the following setting.
We consider an investor who can consume between time $0$ and time $T$,  and denote by $c=(c_t)_{0\leq t\leq T}$ the consumption rate.
We consider a financial market consisting of a bond and one risky asset (for simplicity). Without loss of generality, we assume that the bond is constant. We denote by
$\pi=(\pi_t)_{t\in[0,T]} $ the investment strategy representing the number of shares in the risky asset  $S:=(S_t)$ which evolves according to the Black-Scholes model 
\begin{eqnarray*}
dS_t&=&S_t (\mu dt+ \sigma dW_t\big),\,\,\,S_0=1,
\end{eqnarray*}
where $\mu \in \mathbb{R}$ and $ \sigma>0$.
We denote by $\tilde \Cc$ and $\tilde \Hc$ the following sets
\begin{eqnarray*}
\tilde \Cc&:=&\{ c=(c_t)_{t\in [0,T]}\,\F-\mbox{progressively measurable },\,\, c_t \geq 0\,dt\otimes dP \mbox{ a.e. and } \int_0^T c_tdt <\infty\},\\
\tilde \Hc&:=&\{ \pi=(\pi_t)_{t\in [0,T]}\,\F-\mbox{progressively measurable}, \, \mbox{and} \, \pi \in L(S)\},
\end{eqnarray*}
where $L(S)$ denotes the set of $\F-$ progressively measurable processes such that the stochastic integral with respect to $S$ is well-defined.\\
Given an initial wealth $x\geq 0$ and a policy $(c,\pi)\in \tilde \Cc\times \tilde \Hc$, the wealth process at time $t$ follows the dynamics given by:
\begin{eqnarray}
X_t^{x,c,\pi}&=& x+\int_0^t \pi_u dS_u -\int_0^t c_u du.
\end{eqnarray}
We impose a nonnegativity state constraint~:
\begin{eqnarray*} 
X_t ^{x,c,\pi}&\geq& 0, \; a.s.,  \forall t \in  [0,T].
\end{eqnarray*}
Since we are in the case of  complete market, the budget  constraint is equivalent to $E_{\tilde{P}}[X_T^{x,c,\pi} + \int_0^T c_s ds ] \leqslant x$ where  $\tilde{P}$ is  the unique  equivalent martingale measure.  The investor has preferences modeled by the utility functions $U :=U(c)$ and $\bar U_T:=\bar{U}(X_T^{x,c,\pi})$. The maximization utility  problem of  the investor is given by:
\begin{equation}
\label{RMP}
 v (x):= \sup_{(c,\pi)} Y_0^{x,c,\pi}= \sup_{(c,\pi)} \displaystyle -\bar{\beta} \log E_P[\exp(-\frac{1}{\bar{\beta}}(\bar{\alpha}\bar{U}(X_T^{x,c,\pi})+\int_0^T\alpha U(c_s) ds)].
 \end{equation}
 We assume that  $\delta \equiv 0$, $\alpha = 0$, $\bar \alpha =1$ and $\bar U(z)=\log(z)$. In this case,  we  give a closed formula for  the optimal investment strategy.  Our stochastic control problem \eqref{RMP} is equivalent to
\begin{eqnarray*}
V^{rm}(x):=\sup _{\xi \in \mathcal{X}(x)}E_P\Big[-\exp\Big(- \frac{1}{\bar{\beta}}\bar U(\xi) \Big)\Big],
\end{eqnarray*}
where $\mathcal{X}(x)=\{\xi\geq 0\,,\xi=x+\int_0^T \pi_tdS_t,\,\, \pi\in L(S)\mbox{ and }E_{\tilde P}[\xi]\leq x\}.$
The utility function $U^{rm}(z)=-\exp\Big(- \frac{1}{\bar{\beta}}\bar U(z)\Big)$ is strictly concave and increasing. It satisfies the Inada condition.
 From Kramkov and Schachermayer \cite{krasch99}, the optimal terminal wealth is given by
\begin{eqnarray}\label{wealth}
\xi^*=I^{rm}(y \tilde Z_T)\,\, a.s.
\end{eqnarray}
where
$I^{rm}(z)=((U^{rm})^{'})^{-1}(z)$ and $y=(V^{rm})^{'}(x)$.
 We know by a classical result in the duality theory that
\begin{eqnarray}\label{dynamiquerichesseoptimale}
X_t^{x,\pi^*}=E_{\tilde P}[I^{rm}(y \tilde Z_T)|\Fc_t],\,\,t\in [0,T].
\end{eqnarray}
From the definition of $\bar U$, we have $U^{rm}(x)=-x^{-\frac{1}{\bar{\beta}}}$ which implies
$I^{rm}(z)=(\frac{1}{\bar{\beta}})^{\frac{\bar{\beta}}{1+\bar{\beta}}}z^{-\frac{\bar{\beta}}{1+\bar{\beta}}}$
and so from equation \reff{dynamiquerichesseoptimale}, we deduce that
\begin{eqnarray}\label{dynamiquerichesseoptimale1}
X_t^{x,\pi^*}=(\frac{1}{\bar{\beta}})^{\frac{\bar{\bar{\beta}}}{1+\bar{\beta}}}y^{-\frac{\bar{\beta}}{1+\bar{\beta}}}
E_{\tilde P}[\tilde Z_T^{-\frac{\bar{\beta}}{1+\bar{\beta}}}|\Fc_t],\,\,t\in [0,T].
\end{eqnarray}
The density of the risk neutral measure is given by $\tilde Z_T=\Ec(-\frac{\mu}{\sigma} W_T)$ and
by the Girsanov theorem,
$\tilde W_t=W_t+\frac{\mu}{\sigma}t$ is a $\tilde P$-Brownian motion.
From \reff{dynamiquerichesseoptimale1}, we obtain
\begin{eqnarray}\label{dynamiquerichesseoptimale2}
X_t^{x,\pi^*}&=&
(\frac{1}{\bar{\beta}})^{\frac{\bar{\beta}}{1+\bar{\beta}}}y^{-\frac{\bar{\beta}}{1+\bar{\beta}}}
E_{\tilde P}[\exp\Big(-\frac{\bar{\beta}}{1+\bar{\beta}}(-\frac{\mu}{\sigma}\tilde W_T+\frac{\mu^2}{2\sigma^2}T)\Big) |\Fc_t]\\
&=&(\frac{1}{\bar{\beta}})^{\frac{\bar{\beta}}{1+\bar{\beta}}}y^{-\frac{\bar{\beta}}{1+\bar{\beta}}}
\exp(-\frac{\bar{\beta}}{(1+\bar{\beta})^2}\frac{\mu^2}{2\sigma^2}T)\bar Z_t,\nonumber
\end{eqnarray}
where $\bar Z_t=\Ec(\frac{\bar{\beta}}{1+\bar{\beta}}\frac{\mu}{\sigma}\tilde W_t)$.
Since $X_0^{x,\pi^*}=x$, we have $(\frac{1}{\bar{\beta}})^{\frac{\bar{\beta}}{1+\bar{\beta}}}y^{-\frac{\bar{\beta}}{1+\bar{\beta}}}\exp(-\frac{\bar{\beta}}{(1+\bar{\beta})^2}\frac{\mu^2}{2\sigma^2}T)=x$.
From equation \reff{dynamiquerichesseoptimale2} and using It\^o's formula, we have
\begin{eqnarray*}\label{dynamiquerichesseoptimale3}
dX_t^{x,\pi^*}=x \frac{\bar{\beta}}{1+\bar{\beta}}\frac{\mu}{\sigma}\bar Z_t d\tilde W_t.
\end{eqnarray*}
Since $dX_t^{x,\pi^*}=\pi^*_t\sigma S_td\tilde W_t$, we have by identification that
\begin{eqnarray*}
\pi^*_t&=&x\frac{\bar{\beta}}{1+\bar{\beta}}\frac{\mu\bar Z_t}{\sigma^2 S_t}\\&=&
x\frac{\mu}{\sigma^2}\frac{\bar{\beta}}{1+\bar{\beta}}\frac{\exp({\frac{\bar{\beta}}{1+\bar{\beta}}\frac{\mu}{\sigma}}\tilde{W}_t-\frac{1}{2}\frac{\bar{\beta}^2}{(1+\bar{\beta})^2}t\frac{\mu^2}{\sigma^2})}{ \exp(\sigma \tilde{W}_t-\frac{1}{2}\sigma^2t )}
\\&=&
x\frac{\bar{\beta} \mu}{(1+\bar{\beta})\sigma^2}\exp({(\frac{\bar{\beta}}{1+\bar{\beta}}
\frac{\mu}{\sigma}}-\sigma)\tilde{W}_t-\frac{1}{2}(\frac{\bar{\beta}^2}{(1+\bar{\beta})^2}\frac{\mu^2}{\sigma^2}-\sigma^2)t),
 \quad a.s., \; \forall \, t \in [0,T].
\end{eqnarray*}
When $\beta$ goes to $+\infty$, we force the  penalty term which appears in the dynamics of the value function to vanish and our model converges to the classical utility maximization setting, 
$$\pi^{*,\infty}_t=x
\frac{\mu\bar Z^\infty_t}{\sigma^2 S_t}$$
where $$\bar Z_t^\infty=\Ec(\frac{\mu}{\sigma}\tilde W_t).$$
Such results could be interpreted as a stability result.
\section{Appendix}
\setcounter{equation}{0}
\setcounter{Assumption}{0}
\setcounter{Example}{0}
\setcounter{Theorem}{0}
\setcounter{Proposition}{0}
\setcounter{Corollary}{0}
\setcounter{Lemma}{0}
\setcounter{Definition}{0}
\setcounter{Remark}{0}
\subsection{Proof of the Bellman optimal principle}
\subsubsection{f-divergence case}
\begin{Lemma}\label{lemma3}
 For all $\tau \in \Sc$ and all $Q\in \Qc_f$, the random variable $J(\tau ,Q)$ belongs  to $L^1(Q)$
\end{Lemma}
\begin{proof}
By definition $$J(\tau ,Q)\leq \Gamma(\tau ,Q) \leq \E_Q[|c(.,Q)|| \Fc_\tau],$$
and consequently $$(J(\tau ,Q))^+ \leq \E_Q[|c(.,Q)|| \Fc_\tau]$$
is $Q-$ integrable according to  Proposition \ref{prop1}.
\\Let us show that $(J(\tau ,Q))^-$ is $Q-$ integrable. We fix $Z^{Q'} \in \Dc(Q,\tau).$
In inequality (\ref{ineq7}), choosing $\gamma > 0$ such that $\beta \displaystyle e^{(-T \Vert\delta \Vert _{\infty})}-\frac{1}{\gamma} =0$, then we obtain
\begin{equation}
\begin{split}
\Gamma(\tau,Q')&\geq -B:=-\beta \kappa \frac{1}{Z_\tau^Q}(T \parallel \delta \parallel_\infty+1) -\frac{1}{Z_\tau^Q}\Big[\frac{1}{\gamma}\E_P[f^*(\gamma \alpha \int_{0}^{T}|U(s)|ds+\gamma \bar{\alpha}
|\bar{U}_T| )|\Fc_\tau]\Big].
\end{split}
\end{equation}
Since the random variable $B$ is nonnegative and does not depend on $Q'$, we conclude that
$J(\tau,Q)\geq -B.$
Since $f^*(x)\geq 0 $ for all $x\geq 0$, we have
\begin{equation}
\begin{split}
J(\tau,Q)^-\leq B:= \beta \kappa \frac{1}{Z_\tau^Q}(T \parallel \delta \parallel_\infty+1) +\frac{1}{Z_\tau^Q}\Big[\frac{1}{\gamma}\E_P[f^*(\gamma \alpha \int_{0}^{T}|U(s)|ds+\gamma \alpha'
|U^{\prime }_T| )|\Fc_\tau]\Big].
\end{split}
\end{equation}
Finally, $B\in L^1(Q)$ because the assumption \textbf{(A1)-(A2)}.
\ep
\end{proof}
\begin{Lemma}
The space $\Dc$  is compatible and stable under bifurcation and the cost functional $c$ is coherent.
\end{Lemma}
\begin{proof}
1-We first prove that $\Dc$ is compatible
\\Take $Z^Q\in \Dc, \tau \in \Sc$ and $Z^{Q'}\in \Dc(Q,\tau)$. Then,  from definition of $\Dc(Q,\tau)$ we have $Q|_{\Fc_\tau}=Q'|_{\Fc_\tau}$
\\2- Take  $Z^Q\in \Dc,\tau \in \Sc , A\in \Fc_\tau$ and $Z^{Q'}\in \Dc(Q,\tau)$ again.  The fact that
$Z^Q|\tau_A|Z^{Q'}:=Z^{Q'}\mathbf{1}_A+Z^Q\mathbf{1}_{A^c}$
is still in $\Dc$ must be checked.
\\To this end, it is enough to show that $Z^Q|\tau_A|Z^{Q'}$ is a $\Fc$-martingale and that $(Z^Q|\tau_A|Z^{Q'})_T$ defines a probability measure in $\Qc_f$.
\\Let us start proving that $Z^Q|\tau_A|Z^{Q'}$ is a martingale. Since our time horizon $T$ is finite, we have to prove that
$$\E_P[(Z^Q|\tau_A|Z^{Q'})_T|\Fc_t]=(Z^Q|\tau_A|Z^{Q'})_t .$$
Observing that $\mathbf{1}_{\{\tau\leq t\}}+\mathbf{1}_{\{\tau > t\}}\equiv 1$, we have
\begin{equation*}
\begin{split}
\E_P[(Z^Q|\tau_A|Z^{Q'})_T|\Fc_t]&= \E_P[Z_T^{Q'}I_A+Z_T^Q\mathbf{1}_{A^c}(\mathbf{1}_{\{\tau\leq t\}}+\mathbf{1}_{\{\tau > t\}})|\Fc_t]
\\&=\E_P[Z_T^{Q'}\mathbf{1}_{A\cap \{\tau\leq t\}}|\Fc_t]+\E_P[Z_T^{Q'}\mathbf{1}_{A\cap \{\tau > t\}}|\Fc_t]
\\&+\E_P[Z_T^{Q}\mathbf{1}_{A\cap \{\tau\leq t\}}|\Fc_t]+\E_P[Z_T^{Q}\mathbf{1}_{A\cap \{\tau > t\}}|\Fc_t].
\end{split}
\end{equation*}
Since $A\cap \{\tau\leq t\}$ and $A^c\cap \{\tau\leq t\}$ are in $\Fc_t$, while  $A\cap \{\tau > t\}$ and  $A^c\cap \{\tau > t\}$ are in $\Fc_\tau$ , we have
\begin{equation*}
\begin{split}
&\E_P[(Z^Q|\tau_A|Z^{Q'})_T|\Fc_t]
\\&=I_{A\cap \{\tau\leq t\}}\E_P[Z_T^{Q'}|\Fc_t]+\E_P[\E_P[Z_T^{Q'}\mathbf{1}_{A\cap \{\tau > t\}}|\Fc_{\tau \vee t}]|\Fc_t]
\\&+\mathbf{1}_{A^c\cap \{\tau\leq t\}}\E_P[Z_T^{Q}|\Fc_t]+\E_P[\E_P[Z_T^{Q}\mathbf{1}_{A^c\cap \{\tau > t\}}|\Fc_{\tau \vee t}]|\Fc_t]
\\&=\mathbf{1}_{A\cap \{\tau\leq t\}}Z_t^{Q'}+P[Z_{\tau \vee t}^{Q'}\mathbf{1}_{A\cap \{\tau > t\}}|\Fc_t]
+ \mathbf{1}_{A^c\cap \{\tau\leq t\}}Z_t^{Q}+P[Z_{\tau \vee t}^{Q}\mathbf{1}_{A^c\cap \{\tau > t\}}|\Fc_t]
\\&=\mathbf{1}_{A\cap \{\tau\leq t\}}Z_t^{Q'}+\E_P[Z_\tau ^{Q'}\mathbf{1}_{A\cap \{\tau > t\}}|\Fc_t]
+ \mathbf{1}_{A^c\cap \{\tau\leq t\}}Z_t^{Q}+\E_P[Z_\tau^{Q}\mathbf{1}_{A^c\cap \{\tau > t\}}|\Fc_t].
\end{split}
\end{equation*}
From the definition of $\Dc(Q,\tau)$, we have $Z_\tau^{Q'}=Z_\tau^Q$ and so
\begin{equation}
\begin{split}
&\E_P[(Z^Q|\tau_A|Z^{Q'})_T|\Fc_t]
\\&=\mathbf{1}_{A\cap \{\tau\leq t\}}Z_t^{Q'}+\E_P[Z_\tau ^{Q}\mathbf{1}_{A\cap \{\tau > t\}}|\Fc_t]
+ \mathbf{1}_{A^c\cap \{\tau\leq t\}}Z_t^{Q}+\E_P[Z_\tau^{Q}I_{A^c\cap \{\tau > t\}}|\Fc_t]
\\&=\mathbf{1}_{A\cap \{\tau\leq t\}}Z_t^{Q'}+\E_P[Z_\tau ^{Q}\mathbf{1}_{\{\tau > t\}}|\Fc_t]+ \mathbf{1}_{A^c\cap \{\tau\leq t\}}Z_t^{Q}
\\&=\mathbf{1}_{A\cap \{\tau\leq t\}}Z_t^{Q'}+Z_t ^{Q}\mathbf{1}_{\{\tau > t\}}+ \mathbf{1}_{A^c\cap \{\tau\leq t\}}Z_t^{Q}
\\&=\mathbf{1}_{A\cap \{\tau\leq t\}}Z_t^{Q'}+Z_t ^{Q}(\mathbf{1}_{\{\tau > t\}\cap A}+\mathbf{1}_{\{\tau > t\}\cap A^c})+ \mathbf{1}_{A^c\cap \{\tau\leq t\}}Z_t^{Q}
\\&=\mathbf{1}_{A\cap \{\tau\leq t\}}Z_t^{Q'}+Z_t^{Q'}\mathbf{1}_{\{\tau > t\}\cap A}+Z_t^{Q}\mathbf{1}_{\{\tau > t\}\cap A^c}+ \mathbf{1}_{A^c\cap \{\tau\leq t\}}Z_t^{Q}
\\&=\mathbf{1}_AZ_t^{Q'}+\mathbf{1}_A^cZ_t^Q
\\&=(Z^Q|\tau_A|Z^{Q'})_t.
\end{split}
\end{equation}
From the definition of  $Z^Q|\tau_A|Z^{Q'}$,  we have $Z^Q|\tau_A|Z^{Q'} \in L^1([0,T])$ and so $Z^Q|\tau_A|Z^{Q'}$  is an $\Fc$-martingale which implies
$$\E_P[(Z^Q|\tau_A|Z^{Q'})_T]=\E_P[Z_0^{Q'}\mathbf{1}_A+Z_0^Q\mathbf{1}_{A^c}]=\mathbf{1}_A+\mathbf{1}_{A^c}=1.$$
It remains to show that $d(\bar{Q})<\infty$ where the density of $\bar{Q}$ is given by $Z^Q|\tau_A|Z^{Q'}.$
We have
\begin{equation}
\begin{split}
d(\bar{Q}|P)+\kappa&=\E_P[f(Z_T^{Q'}\mathbf{1}_A+Z_T^Q\mathbf{1}_{A^c})+\kappa]
\\& = \E_P[\mathbf{1}_A(f(Z_T^{Q'})+\kappa)+\mathbf{1}_{A^c}(f(Z_T^Q)+\kappa)]
\\& \leq \E_P[(f(Z_T^{Q'})+\kappa)+(f(Z_T^Q)+\kappa)]
\\& \leq d(Q|P)+d(Q'|P)+2\kappa.
\end{split}
\end{equation}
The first inequality is deduced from  assumption \textbf{(H2)}.
Then $$d(\bar{Q}|P) \leq d(Q|P)+d(Q'|P)+\kappa< \infty.$$
\\3- Take $Z^Q$ and $Z^{Q'}$ in $\Dc$, we denote by $A$ the set $\{\omega ; Z_T^Q(\omega)=Z_T^{Q'}(\omega)\}$.  It must be proven that
$$ c(\omega,Z^Q(\omega))=c(\omega,Z^{Q'}(\omega))$$ on $A$   $Q-$a.s and ${Q'}-$a.s respectively.
\ep
\end{proof}
\subsubsection{Consistent time penalty case}
\begin{Lemma}
The space $\Dc^c $ is compatible, stable under bifurcation and the cost functional $c$ is coherent.
\end{Lemma}
\noindent
\begin{proof}
\\ 1. $\Dc^c $ is compatible: let $\eta \in \Dc, \tau \in \Sc$ and  $\eta'\in \Dc^c(Q^\eta,\tau).$ Then, by definition of $ \Dc^c(Q^\eta, \tau) $ we have $ Q^\eta|_{\Fc_\tau} = Q^{\eta'}|_{\Fc_\tau}.$
\\ 2.  $\Dc^c$ is stable under bifurcation: let again $\eta \in \Dc^c,\tau \in \Sc , A\in \Fc_\tau$ and $\eta'\in \Dc^c(Q^\eta,\tau).$ It must be checked that
$\eta''=\eta|\tau_A|\eta':=\eta \mathbf{1}_A+\eta' \mathbf{1}_{A^c}$
remains in $\Dc^c. $ i.e.  $E_{Q ^{\eta''}}[\dint_0^Th (\eta''_u) du] <+\infty $. Indeed,
\begin{equation*}
\begin{split}
E_{Q^{\eta''}}[\int_0^Th(\eta''_u)du]&\leq E_{Q^{\eta''}}[\mathbf{1}_A\int_0^Th(\eta_u)du+\mathbf{1}_{A^c}\int_0^Th(\eta'_u)du]
\\&=E_{P}[Z_T^{\eta''}\mathbf{1}_A\int_0^Th(\eta_u)du+Z_T^{\eta''}\mathbf{1}_{A^c}\int_0^Th(\eta'_u)du]
\\&=E_{P}[Z_T^{\eta}\mathbf{1}_A\int_0^Th(\eta_u)du+Z_T^{\eta'}\mathbf{1}_{A^c}\int_0^Th(\eta'_u)du]
\\&\leq E_{Q^\eta}[\int_0^Th(\eta_u)du]+E_{Q^{\eta'}}[\int_0^Th(\eta'_u)du].
\end{split}
\end{equation*}
The last inequality is deduced from the non negativity of $h$ and the second equality is deduced from the definition of $\eta'.$
\\
3. The cost function $c$ is coherent:  let $\eta$ and $\eta'$ in $\Dc^c:$ denote by $ A $ the set $\{\omega, \eta(\omega) = \eta'(\omega) \} $. It is obvious that
$$ c(\omega, \eta (\omega)) = c (\omega, \eta' (\omega)) $$   $Q-a.s$ and $ {Q'}-a.s $ on $A$.
\ep
\end{proof}
\begin{Lemma}
For all $ \tau \in \Sc $ and $ Q^\eta \in \Qc^c_f $ , the random variable $J(\tau, Q^\eta) $ is in $ L^1(Q^\eta).$
\end{Lemma}
\begin{proof}
By definition, we have $$ J(\tau, Q^\eta) \leq \Gamma (\tau, Q^\eta) \leq \E_ {Q^\eta} [|c(.,Q^\eta)| | \Fc_\tau],$$
which implies that $$ (J(\tau, Q^\eta))^+\leq \E_{Q ^\eta} [|c(.,Q^\eta)| | \Fc_\tau] $$ and so
$ (J(\tau, Q^\eta))^+$ is $ Q^\eta $-integrable by Proposition \ref{prop3.1}.
\\It remains to show that  $(J(\tau, Q^\eta))^-$ is also $Q^\eta $-integrable.
\\ Fix $\eta'\in \Dc^c (Q^\eta, \tau).$
we have:
\begin{equation*}
\begin{split}
\beta E_{Q^{\eta'}}[\dint_0^T \delta_s S_s^\delta (\int_0^s h(\eta'_u)du) ds + S_T^\delta \int_0^T h(\eta'_s)ds|\Fc_\tau]& \geq \beta E_{Q^{\eta'}} [S_T^\delta \int_\tau^T h(\eta'_s)ds|\Fc_\tau]\\& \geq \beta e^{-\|\delta \|_{\infty}T} \gamma_\tau (Q^{\eta'}).
\end{split}
\end{equation*}
Moreover, since $0 \leq S^\delta \leq 1 $ and using Bayes' formula, we have:
\begin{equation*}
\E_{Q^{\eta'}} [\Uc_{0,T}^\delta|\Fc_\tau]\geq -\E_{Q^{\eta'}} [R|\Fc_\tau] =-\frac{1}{Z_\tau^{\eta}} \E_P[Z_T^{\eta'} R|\Fc_\tau].
\end{equation*}
\\Using the inequality (\ref{convexinequality}) and Bayes' formula, we get:
\begin{equation*}
\begin{split}
\E_P[Z_T^{\eta'} R|\Fc_\tau] & \leq \frac{1}{\lambda} \E_P [Z_T^{\eta'}\log Z_T^{\eta'} +e^{-1} e^{\lambda R}|\Fc_\tau]
\\ & = \frac {1}{\lambda} Z_\tau^{\eta'}\E_{Q^{\eta'}} [\log Z_T^{\eta'}|\Fc_\tau]  + \frac {e^{-1}} {\lambda} \E_P [e^{\lambda R}|\Fc_\tau]
\\ & = \frac {1}{\lambda} Z_\tau^{\eta'}\E_{Q^{\eta'}} [\int_0^T\eta'_udW_u-\frac{1}{2}\int_0^T|\eta'|^2_udu|\Fc_\tau]  + \frac {e^{-1}} {\lambda} \E_P [e^{\lambda R}|\Fc_\tau]
\\ & = \frac {1}{\lambda} Z_\tau^{\eta'}(\int_0^\tau\eta'_udW_u-\frac{1}{2}\int_0^\tau|\eta'|^2_udu+\E_{Q^{\eta'}} [\int_\tau^T\eta'_udW_u-\int_\tau^T|\eta'|^2_udu|\Fc_\tau]
\\&+ \E_{Q^{\eta'}}[\frac{1}{2}\int_\tau^T|\eta'|^2_udu|\Fc_\tau])+\frac {e^{-1}} {\lambda} \E_P [e^{\lambda R}|\Fc_\tau].
\end{split}
\end{equation*}
By the Girsanov theorem the process $(\int_0^.\eta'_udW_u-\int_0^.|\eta'|^2_udu)$ is a local $Q^{\eta'}$- martingale and therefore $\E_{Q^{\eta'}} [\int_\tau^T\eta'_udW_u-\int_\tau^T|\eta'|^2_udu|\Fc_\tau]=0.$ Consequently, by using that $Z_\tau^{\eta'}=Z_\tau^{\eta}$, we have
\begin{equation*}
\begin{split}
&\E_P[Z_T^{\eta'} R|\Fc_\tau]  \leq  \frac {1}{\lambda} Z_\tau^{\eta}\Big(\int_0^\tau\eta'_udW_u-\frac{1}{2}\int_0^\tau|\eta'|^2_udu
+\E_{Q^{\eta'}}[\frac{1}{2}\int_\tau^T|\eta'|^2_udu|\Fc_\tau]\Big)+\frac {e^{-1}} {\lambda} \E_P [e^{\lambda R}|\Fc_\tau]
\\& =\frac {1}{\lambda} Z_\tau^{\eta}\Big(\int_0^\tau\eta_udW_u-\frac{1}{2}\int_0^\tau|\eta|^2_udu
+\E_{Q^{\eta'}}[\frac{1}{2}\int_\tau^T|\eta'|^2_udu|\Fc_\tau]\Big)+\frac {e^{-1}} {\lambda} \E_P [e^{\lambda R}|\Fc_\tau]
\\ & \leq \frac {1}{\lambda} Z_\tau^{\eta}\Big(\int_0^\tau\eta_udW_u-\frac{1}{2}\int_0^\tau|\eta|^2_udu
+\E_{Q^{\eta'}}[\frac{1}{2}\int_\tau^T\frac{h(\eta'_u)+\kappa_2}{\kappa_1}du|\Fc_\tau]\Big)+\frac {e^{-1}} {\lambda} \E_P [e^{\lambda R}|\Fc_\tau].
\end{split}
\end{equation*}
Thus, we have
\begin{equation*}
\begin{split}
\E_{Q^{\eta'}} [\Uc_{0,T}^\delta|\Fc_\tau]&\geq -\frac {1}{\lambda} \Big(\int_0^\tau\eta_udW_u-\frac{1}{2}\int_0^\tau|\eta|^2_udu
+\E_{Q^{\eta'}}[\frac{1}{2}\int_\tau^T\frac{h(\eta'_u)+\kappa_2}{\kappa_1}du|\Fc_\tau]\Big)\\&-\frac {e^{-1}} {\lambda}\frac{1}{Z_\tau^{\eta}} \E_P [e^{\lambda R}|\Fc_\tau],
\end{split}
\end{equation*}
and consequently
\begin{equation*}
\begin{split}
\Gamma(\tau,Q^{\eta'})\geq
&\beta e^{-\|\delta \|_{\infty}T} \gamma_\tau (Q^{\eta'})-\frac{1}{\lambda}\E_{Q^{\eta'}}[\frac{1}{2}\int_\tau^T\frac{h(\eta'_u)+\kappa_2}{\kappa_1}du|\Fc_\tau]
\\&-\frac {1}{\lambda} (\int_0^\tau\eta_udW_u-\frac{1}{2}\int_0^\tau|\eta|^2_udu)
 -\frac {e^{-1}} {\lambda}\frac{1}{Z_\tau^{\eta}} \E_P [e^{\lambda R}|\Fc_\tau].
 \\&=(\beta e^{-\|\delta \|_{\infty}T}-\frac{1}{2\lambda\kappa_1}) \gamma_\tau (Q^{\eta'})
\\&-\frac {1}{\lambda} (\int_0^\tau\eta_udW_u-\frac{1}{2}\int_0^\tau|\eta|^2_udu
 +\frac{T\kappa_2}{2\kappa_1})-\frac {e^{-1}} {\lambda}\frac{1}{Z_\tau^{\eta}} \E_P [e^{\lambda R}|\Fc_\tau].
\end{split}
\end{equation*}
Let $ \lambda > 0 $ such that $\beta e^{-\|\delta \|_{\infty}T}-\frac{1}{2\lambda\kappa_1}=0 $ then
\begin{equation*}
\Gamma(\tau,Q^{\eta'})\geq -\frac {1}{\lambda} (\vert\int_0^\tau\eta_udW_u \vert+\frac{1}{2}\int_0^\tau|\eta|^2_udu
 +\frac{T\kappa_2}{2\kappa_1})-\frac {e^{-1}} {\lambda}\frac{1}{Z_\tau^{\eta}} \E_P [e^{\lambda R}|\Fc_\tau]:=-B.
\end{equation*}
Since the random variable $B$ is nonnegative and does not depend on $Q^{\eta'}$, we conclude that
$J(\tau,Q)\geq -B.$ So that $J(\tau,Q)^-\leq B.$
\\It thus remains to be shown that $B\in L^1(Q^\eta).$
\\Under assumptions \textbf{(A1)-(A'2)}, we have
\begin{equation*}
E_{Q^\eta}[\frac{1}{Z_\tau^{\eta}} \E_P [e^{\lambda R}|\Fc_\tau]]=E_P[\E_P [e^{\lambda R}|\Fc_\tau]]=\E_P [e^{\lambda R}]<+\infty.
\end{equation*}
Moreover
\begin{equation*}
E_{Q^\eta}[|\int_0^\tau\eta_udW_u|+\frac{1}{2}\int_0^\tau|\eta|^2_udu]<+\infty.
\end{equation*}
Hence, $B\in L^1(Q^\eta).$
\ep
\end{proof}


\begin{thebibliography}{10}

\bibitem{AHS03}
 Anderson,  E.,  Hansen,  L.P.,   Sargent, T.J.,
\newblock  A quartet of semigroups for model specification, robustness, prices of risk, and model detection.
\newblock {\em Journal of the European Economic Association}, Vol. 1, 68-123, 2003.
\bibitem{Bar09}
Barrieu P., El Karoui N.,
\newblock {\em Indifference Pricing, Theory and Application}, chapter Pricing,
  Hedging, and Designing Derivatives with Risk Measures, pages 77-146.
\newblock Princeton University Press, 2009.
\bibitem{Bar13} Barrieu P., El Karoui N.,
\newblock Monotone stability of quadratic semimartingales with applications to unbounded general quadratic BSDEs
\newblock{\em The Annals of Probability}, Vol. 41(3B), 1831-1863, 2013.
\bibitem{BMS05}
Bordigoni G., Matoussi A., Schweizer M.,
\newblock A stochastic control approach to a robust utility maximization
  problem.  \newblock {\em Abel Symposium 2005. Stochastic Analysis and Applications, eds. F.E. Benth, G. Di Nunno T., Lindstrom B., Oksendal T. Zhang. Springer-Verlag Berlin}, pages 125-151, 2007.
\bibitem{BH07}
Briand Ph., Hu, Y.
\newblock Quadratic BSDE with convex generators and unbounded terminal
  conditions.
\newblock {\em Probability Theory Related Fields,} Vol. 141, 543-567, 2008.
\bibitem{BUR05}
Burgert C., Rüschendorf L.,
\newblock Optimal consumption strategies under model uncertainty.
\newblock{\em Statist. Decisions,} Vol. 23, 1-14, 2005.
\bibitem{CH89}
Cox, J. , Huang, C.
\newblock Optimal consumption and portfolio policies when asset prices follow a
  diffusion process.
\newblock {\em Journal of Economic Theory}, Vol. 49, 33-83, 1989.
\bibitem{Ciz}
Csiszar, I.
\newblock Information-type measures of difference of probability distributions
  and indirect.
\newblock {\em Stud. Sci. Math. Hung}, Vol. 2, 299-318, 1967.
\bibitem{Delb09} Delbaen F.,  Hu Y.,    Bao X.,
\newblock Backward stochastic differential equations with superquadratic growth
\newblock {\em Probab. Theory Related Fields,} Vol. 150, 145-192, 2011.
\bibitem{Delb08}
Delbaen F., Peng S., Rosazza Gianin E.,
\newblock Representation of the penalty term of dynamic concave utilities.
\newblock{\em Finance Stoch.,} Vol. 14, 449-472, 2010.
\bibitem{DE92}
Duffie D., Epstein L.
\newblock  Asset Pricing with Stochastic Differential Utility.
\newblock {\em Review of Financial Studies} Vol. 5,411-436, 1992.
\bibitem{ELK81}
El Karoui N.,
\newblock Les aspects probabilistes du contrôle stochastique.
\newblock {\em Ecole d' été de Probabilités de Saint Flour IX, Lecture Notes in
  Mathematics},Vol.  876,73-238, 1981.
  \bibitem{EZ89}
Epstein L., Zin E.
\newblock Substitution, Risk Aversion, and the Temporal Behavior of Consumption and Asset Returns: A Theoretical
Framework.
\newblock {em Econometrica} Vol. 57, 937-969, 1989.
\bibitem{FG06} Föllmer, H., Gundel, A. 
\newblock Robust projections in the class of martingale measures.
\newblock {\em Illinois Journal of Mathematics}, 
Vol. 50(2),  439-472, 2006.
 \bibitem{Fri00}
Frittelli M.,
\newblock The minimal entropy martingale measure and the valuation problem in incomplete markets
\newblock {\em Mathematical Finance}, Vol. 10(1), 39-52.
\bibitem{GS89}
Gilboa I., Schmeidler D.,
\newblock Maxmin expected utility with a non-unique prior.
\newblock {\em Journal of Mathematical Economics},Vol. 18, 141-153, 1989.
\bibitem{HS01}
Hansen L., Sargent T.,
\newblock Robust control and model uncertainty,
\newblock{\em Amer.Econom. Rev.,} Vol. 91, 60-66, 2001.
\bibitem{HSTW06}
Hansen, L. P.,  Sargent, T. J.,  Turmuhambetova, G.,  Williams, N.,
\newblock  Robust control and model misspecification.
\newblock{\em  J. Econom. Theory},  Vol. 128(1), 45-90, 2006.
\bibitem{HIM06}
Hu, Y., Imkeller, P., Müller M.,
\newblock  Utility maximization in incomplete markets.
\newblock{\em  The Annals of Applied  Probability},  Vol. 15(3), 1691-1712, 2006.
\bibitem{JST05}
Jouini E., Schachermayer W., Touzi N.,
\newblock Law Invariant Risk Measures have the Fatou Property.
\newblock {\em Advances in Mathematical Economics}, Vol. 9, 49-72, 2006.
\bibitem{Kar91}
Karatzas I., Lehoczky J.P., Shreve S., Xu G.,
\newblock Martingale and duality methods for utility maximization in an
  incomplete market.
\newblock {\em SIAM Journal on Control and Optimization}, Vol. 29, 702-730, 1991.
\bibitem{krasch99}
Kramkov,  D. and  Schachermayer, W.,
\newblock The Asymptotic Elasticity of Utility
Functions and Optimal Investment in Incomplete Markets. \newblock {\em The Annals of Applied Probability}, 9, 904-950, 1999.
\bibitem{KP78} Kreps D. and  Porteus E.
\newblock Temporal Resolution of Uncertainty and Dynamic Choice Theory.
\newblock {\em Econometrica},Vol. 46, 185-200, 1978.
\bibitem{LS12} Laeven R.J.A., Stadje M.A.,
\newblock Robust Portfolio choice and indifference valuation,
\newblock{\em Mathematics of Operations Research,}Vol. 39(4), 1109-1141, 2014.
\bibitem{Mac06}
Maccheroni F.,  Marinacci M., Rustichini A.,
\newblock Ambiguity aversion,robustness,and the variational representation of
  preferences.
\newblock {\em Econometrica}, Vol.74(6), 1447-1498, 2006.
\bibitem{MMM14}
Matoussi A.,  Mezghani H., Mnif M.,
\newblock Robust Utility Maximization Under Convex Portfolio Constraints.
\newblock {\em Applied Mathematics \& Optimization},Vol. 71, 313-351, 2015.
\bibitem{Mer71}
Merton R.,
\newblock Optimum consumption and portfolio rules in a continuous-time model.
\newblock {\em Journal of Economic Theory},Vol. 3, 373-413, 1971.
\bibitem{QUEN04}
Quenez M.,
\newblock Optimal portfolio in a multiple-priors model.
\newblock{\em Seminar on Stochastic Analysis,
Random Fields and Applications IV}, 291-321, Progr. Probab., 58, Birkhauser, Basel.
\bibitem{Rao}
Rao M.M., Ren Z.D.,
\newblock Theory of Orlicz spaces
\newblock {\em volume 146 of Pure and Applied Mathematics.
Marcel Dekker, Inc.}, 1991.
\bibitem{RElk00}
Rouge, R., El Karoui, N.,
\newblock  Pricing via utility maximization and entropy,
\newblock {\em Math. Finance}, Vol.  10(2), 259?276, 2000.
\bibitem{SC08}
Schied A.,
\newblock Robust optimal control for a consumption-investment problem.
\newblock {\em Mathematical Methods of Operations Research}, Vol. 67(1), 2008.
\bibitem{SC07}
Schied A.,
\newblock  Optimal investments for risk- and ambiguity-averse preferences: a duality approach.
\newblock {\em Finance Stoch.} Vol. 11(1), 107-129, 2007.
\bibitem{SCW05}
Schied  A., Wu C.T.,
\newblock  Duality theory for optimal investments under model uncertainty,
\newblock{\em Stat.Decisions} Vol. 23(3), 199-217,2005
\bibitem{Schmei89}
Schmeidler D.,
\newblock Subjective probability and expected utility without additivity.
\newblock {\em Econometrica}, Vol. 57, 571-587, 1989.
\bibitem{SS03}
Schroder  M. , Skiadas C.,
\newblock Optimal lifetime consumption-potfolio strategies under trading
  constraints and generalized recursive preferences.
\newblock {\em Stochastic processes and their applications}, Vol. 108, 155-202,
  2003.
\bibitem{SK03}
Skiadas C.,
\newblock Robust control and recursive utility.
\newblock {\em Finance and Stochastics},Vol. 7, 475-489, 2003.
\bibitem{WIT06}
Wittmüss W.,
\newblock Robust optimization of consumption with random endowment.
\newblock{\em Stochastics An International Journal of Probability and Stochastic Processes},
Vol.80, 459-475,2008.
\bibitem{VnM44} Von Neumann, J., Morgenstern, O.,
\newblock{ Theory of Games and Economic Behavior}.
\newblock{\em Princeton University Press, Princeton, New Jersey}, 1944.


\end{thebibliography}
\end{document}